\documentclass{amsart}

\usepackage{amsmath}
\usepackage{amsfonts}       
\usepackage{amsthm}
\usepackage{amssymb}
\usepackage{xcolor}
\usepackage{color}
\usepackage{multicol}
\usepackage{mathtools}
\usepackage{soul}
\usepackage[pdftex]{graphicx}
\usepackage{caption}
\usepackage{float}
\usepackage{epsfig}
\usepackage{pstricks,pst-grad}
\usepackage{graphicx}
\usepackage{todonotes}
\usepackage{nicefrac}
\usepackage{cases}
\usepackage{bm}

\usepackage{pgfplots}
\pgfplotsset{compat=1.17}
\usepgfplotslibrary{fillbetween} 

\newtheorem{theorem}{Theorem}[section]
\newtheorem{lemma}[theorem]{Lemma}

\theoremstyle{definition}

\theoremstyle{remark}

\newtheorem*{theorem*}{Theorem}

\newcommand{\be}{\begin{equation}}
\newcommand{\ee}{\end{equation}}
\newcommand{\ba}{\begin{align}}
\newcommand{\ea}{\end{align}}
\newcommand{\ban}{\begin{align*}}
\newcommand{\ean}{\end{align*}}
\newcommand{\la}{\label}

\newcommand{\medcup}{\mathbin{\scalebox{1.5}{\ensuremath{\cup}}}}

\newcommand{\xo}{\xi_1}
\newcommand{\xt}{\xi_2}
\newcommand{\eo}{\eta_1}
\newcommand{\et}{\eta_2}
\newcommand{\kbg}{k_{\beta,\gamma}}

\newcommand{\kg}{k_{\beta,\gamma}}

\newcommand{\R}{\mathbb{R}}
\newcommand{\C}{\mathbb{C}}

\newcommand{\darr}[4]{{\left\{\begin{array}{ll}
{#1}&{#2}\\[0.2cm]
{#3}&{#4}
\end{array}\right.}}

\def\g{{g_{\lambda}}}

\def\g{{g_{\lambda}}}

\def\g1{{g_{1/2}}}

\newcommand{\ia}{({\rm i})}
\newcommand{\ib}{({\rm ii})}

\title{The fundamental solution of a family of anisotropic fourth-order 
parabolic equations: small-time asymptotics and related Gaussian estimates}

\date{}

\author{G. Barbatis}
\address{Department of Mathematics, National and Kapodistrian University of Athens, \newline
15784 Athens, Greece}
\email{gbarbatis@math.uoa.gr}

\author{E. Karafylis}
\address{Department of Mathematics, National and Kapodistrian University of Athens, \newline
15784 Athens, Greece}
\email{sotikar@math.uoa.gr; sot1kar@gmail.com}

\dedicatory{Dedicated to the memory of E.B. Davies}

\begin{document}
\date{\today}

\vspace{11pt}

\begin{abstract}
We study the fundamental solutions of a class of uniformly elliptic 
fourth-order parabolic equations in two dimensions, whose principal 
symbols are not necessarily strongly convex. For constant-coefficient 
operators, we derive precise small-time asymptotic formulas along certain 
special directions. These formulas are obtained by the method of 
steepest descent and reveal the dependence of the asymptotic behavior on 
the coefficients of the symbol. We subsequently consider operators with 
bounded measurable coefficients, subject to Dirichlet boundary 
conditions, and establish Gaussian bounds for their fundamental solutions 
by means of Davies' exponential perturbation method. The exponential 
constant in these bounds is shown to be sharp, as a consequence of the 
constant-coefficient asymptotics.
\end{abstract}



\maketitle

\section{Introduction}

In this article we study two related problems concerning
a class of anisotropic fourth order linear parabolic equations in two 
dimensions. The associated elliptic operator is assumed to be uniformly 
elliptic and to have, in the variable coefficients case, a symbol of the 
form
\be
 A(x,\xi) =\xi_1^4 + 4\gamma(x) \,
 \xi_1^3 \, \xi_2 +  2\beta(x)\, \xi_1^2 \, \xi_2^2 
  +4\gamma(x) \, \xi_1 \, \xi_2^3 +\xi_2^4 \; .
\la{pr:sy}
\ee
The first problem concerns the small-time asymptotics of the fundamental
solution when the operator has constant coefficients and
acts on $L^2(\R^2)$. In this case the fundamental solution is given by
\[
G(x,t) = \frac{1}{(2\pi)^2} \int_{\R^2}e^{i \, x\cdot \xi -tA(\xi)}d\xi
\]
and the steepest descent method is applied to obtain the asymptotic
formula.
The second problem concerns pointwise Gaussian bounds on
the fundamental solution in the case of uniformly elliptic operators
with bounded measurable coefficients acting on $L^2(\Omega)$; in case
$\Omega \subsetneqq \R^2$ Dirichlet boundary conditions are imposed
on $\partial\Omega$. We apply  Davies' exponential perturbation technique 
and obtain the sharp constant in the exponential factor, where the 
sharpness follows from the asymptotics of the first problem.

Small-time asymptotics for higher-order operators in dimension
larger than one go back to the work of Fedoryuk \cite{F1}.
Later Evgrafov et al. \cite{EP1} considered constant coefficient
operators of any order $2m$ acting on $L^2(\R^d)$ for any dimension
$d\geq 2$ and obtained small-time asymptotics
under the key hypothesis, introduced
therein, of the strong
convexity of the symbol $A(\xi)$; this amounts to a certain auxiliary 
matrix being positive definite.
The main asymptotic formula of \cite{EP1} reads
\[
G(x,t) \; \sim \; \exp\Big\{ - \sigma_m\frac{ 
p(x)^{\nicefrac{2m}{2m-1}}}{t^{\nicefrac{1}{2m-1}}}\Big\} \; , \qquad 
\mbox{ as }t\to 0+ \, ,
\]
where $x\in\R^d\setminus\{0\}$ is fixed and
subexponential terms have been omitted for the sake of
simplicity. Here
$p(x)$ denotes the Finsler-type norm defined by
\be
p(x) =\max_{\xi \neq 0} \frac{x\cdot \xi}{ A(\xi)^{\nicefrac{1}{2m}}},
\qquad x\in\R^d ;
\la{fin}
\ee
this is the natural norm for the formulation of
results involving the operator with symbol $A(\xi)$.
The exponential constant $\sigma_m$ is given explicitly, depends only
on $m$ and for $m=2$ it is
\be
\sigma_2 = 3 \cdot 2^{-\nicefrac{11}{3}} . 
\la{sig2}
\ee
In the subsequent article \cite{EP2} it was shown that strong convexity
cannot be relaxed to convexity, disproving a conjecture from
\cite{EP1}.\footnote{We also note the 
article \cite{LW} which was published much later than \cite{EP1} and 
where the special case of the polyharmonic operator $(-\Delta)^m$ was 
considered.}
Tintarev \cite{T} generalized the result of \cite{EP1}
obtaining small-time asymptotics for the fundamental solution for 
equations with variable smooth coefficients with strongly 
convex symbols. In \cite{Ste}  small-time
asymptotics as well as Gaussian estimates were obtained for
perurbations  by a bounded potential of the bilaplacian in
$\R^3$.
For the large-time behavior of the fundamental solution and the
related problem of eventual positivity
see \cite{DGM,Da95b,GGS,RS-C2} and references therein.

As remarked in \cite{T}, strong convexity is a rather strong assumption.
Small-time asymptotics without the strong convexity hypothesis
were obtained in \cite{BB} where a class of fourth order operators 
in two dimensions was considered.
The associated symbols were of the form
\be
 A(\xi) =\xi_1^4 +2\,\beta\,\xi_1^2\,\xi_2^2 + \xi_2^4 \; , \qquad
 \xi\in\R^2.
\la{bb}
\ee
where $\beta>-1$ by ellipticity\footnote{The more general symbol
$\alpha\xi_1^4 +2\,\beta\,\xi_1^2\,\xi_2^2 + \gamma\xi_2^4$ reduces
to \eqref{bb} by a 
simple change of variables.}. This symbol is strongly convex
when $0<\beta<3$, so the  complementary cases $\beta\in (-1,0)$ and
$\beta\in (3,\infty)$ were considered.

In the present article we continue the study of non-strongly convex
symbols and consider symbols of the more general form
\[
 A(\xi) =\xi_1^4 +4\,\gamma \,\xi_1^3\,\xi_2 +2\,\beta\,\xi_1^2\,\xi_2^2 + 4\,\gamma\, \xi_1\,\xi_2^3 +  \xi_2^4 \; , \qquad \xi\in\R^2.
\]
We focus on the non-strongly convex case and obtain
small-time asymptotic formulas for any choice of $\beta,\gamma\in\R$ for 
which the operator is elliptic. For each such 
choice there are two specific lines where the point
$x\in\R^2\setminus\{0\}$ must belong. This is due to extremal properties 
of the fundamental inequality
\[
{\rm Re}\, A(\xi +i\eta) \geq - k A(\eta) 
\, , \qquad \xi, \, \eta\in\R^2,
\]
where $k$ is the best possible constant. The extremal
points $\eta\in\R^2$, i.e. those for which there exists $\xi\in\R^2$
such that equality holds, lie on two prefered directions.
This differs from the strongly convex case where each $\eta\in\R^2$ 
is extremal.
(Nevertheless, the fundamental inequalities obtained,
cf.  Lemmas \ref{lem:E1} and \ref{lem:E2},
are sufficient in order to establish sharp Gaussian estimates 
that are valid for all $x,x'\in\Omega$ (and $t>0$).) As it turns out,
the exponential constant 
$\sigma_{\beta,\gamma}$ in the resulting asymptotics 
depends in a rather involved way on the
parameters $\beta,\gamma$. This is in contrast to the rigid character
of strong convexity where the exponential constant depends
only on the order.

Gaussian estimates for higher-order operators with bounded measurable
coefficients under Dirichlet boundary conditions were first 
obtained by Davies in \cite{Da95a} by adapting to the higher-order case
his well known exponential perturbation technique.
In the case of a fourth order operator in two dimensions and in the 
presence of lower order terms the estimate has the form
\be
|G(x,x',t)|\leq c_1 t^{-\frac{1}{2}}
\exp\Big(-c_2\frac{|x-x'|^{4/3}}{t^{1/3}}+c_3t\Big),
\la{gauss}
\ee
for some positive constants $c_1$, $c_2$, $c_3$ and all $t>0$
and $x,x'\in\Omega$.

A sharp estimate of the constant $c_2$ for perturbations of
the bilaplacian was established in \cite{BD}. Gaussian estimates
for general operators with strongly convex
symbols\footnote{In this context strong convexity means that
the matrix referred to earlier is positive semi-definite,
as opposed to positive definite which was assumed in \cite{EP1}.}
were obtained in \cite{B2001} where the inequality
\[
|G(x,x',t)|\leq c_{\epsilon} t^{-\frac{1}{2}}\exp\Big\{
-\big( \sigma_2 -D -\epsilon\big) \frac{d_M(x,x')^{4/3}}{t^{1/3}}+
c_{\epsilon,M}t\Big\},
\]
was established for arbitrary $\epsilon $ and $M$ positive. Here
$D\geq 0$ is a constant proportional to the distance in the $L^{\infty}$
norm of the coefficients from $W^{m-1,\infty}$,
$d_M(x,y)$, $M>0$, is a family of Finsler-type distances on $\Omega$ 
which increase with $M$ and converge as $M\to +\infty$ to a Finsler-type 
distance $d(x,y)$. Sharp Gaussian estimates without the
strong convexity assumption were obtained in \cite{BB}.

The literature on Gaussian bounds for higher-order 
operators is large; see, e.g, \cite{AQ,Da97,K,RS-C1}.
Gaussian estimates involving Schrödinger operators with singular 
potentials were obtained in \cite{B2017,CYZ,DDY}. For results on
non-Euclidean settings see \cite{AM,Du,tER,S-C,Str1,Str2}.
We also refer to the articles \cite{AN,LSZ}
for recent results where the exponential perturbation 
technique is extended to the case of time-dependent coefficients.

We close this introduction with one remark. As pointed out
in \cite{EP1}, in the special case of a fourth order operator in two 
dimensions strong convexity is equivalent to convexity, in the sense of 
the Hessian matrix being positive definite (except at the origin). 
Nevertheless we have chosen not to replace the term `strongly convex' by 
`convex' in order to emphasize the importance of strong convexity in the 
general case  of an operator of order $2m$ acting on $L^2(\R^d)$.



\section{The main setting}

We consider the homogeneous polynomial
\be
A(\xi)=\xo^4 + 4\gamma \,\xo^3 \xt + 2\beta \, \xo^2 \xt^2 +4\gamma \,
\xo \xt^3 + \xt^4, \qquad \xi\in\R^2.
\la{syl}
\ee
It is straightforward to see that the set of all $(\beta,\gamma)\in\R^2$ 
for which $A(\xi)$ is positive definite is
\[
 E=\Big\{   -1<\beta\leq 3 \mbox{ and } |\gamma|<\frac{\beta+1}{4} \Big\} \medcup  \Big\{ \beta\geq 3 \mbox{ and }
\gamma^2<\frac{\beta-1}{2}\Big\} \, ;
\]
hence we shall always assume that $(\beta,\gamma)\in E$.
We denote by $H$ the associated elliptic
operator, written formally as $H=A(i\nabla)$, and consider the
fourth-order parabolic problem
\[
\left\{
\begin{array}{ll}
u_t =-Hu , &  \quad x\in\R^2, \;\; t>0,  \\
u|_{t=0} =u_0 , & \quad x\in\R^2,
\end{array}
\right.
\]
where $u_0\in L^2(\R^2)$. 
The solution to this problem is represented by means of the
parabolic Green function given by
\[
G(x,t)=\frac{1}{(2\pi)^2}\int\limits_{\mathbb{R}^2} e^{i\xi \cdot x - A(\xi)t}\,d \xi, \qquad x\in \mathbb{R}^2, \; t>0 .
\]
With a simple change of variables we obtain
\be
G(x,t)= \frac{1}{(2\pi)^2 (4t)^{\nicefrac{2}{3}}}F(\lambda) \;
, \qquad \lambda=(4t)^{-\nicefrac{1}{3}},
\la{g:GF}
\ee
where
\be
F(\lambda)=\int\limits_{\mathbb{R}^2} e^{\lambda\big(i\xi \cdot x -\frac{1}{4} A(\xi)\big)}\,d \xi \, ;
\la{fl}
\ee
this is the integral to which the method of steepest descent will be applied
and for this we shall need to consider the holomorphic function
\be
\phi (z) = ix \cdot z - \frac{1}{4}A(z) = ix \cdot z - \frac{1}{4}
\Big(z_1^4+4 \gamma z_1^3z_2+2\beta z_1^2z_2^2+4 \gamma z_1z_2^3+z_2^4\Big)
\la{PHI}
\ee
of two complex variables $(z_1,z_2)\in\C^2$.

As mentioned, the natural way to express the resulting asymptotics is 
with respect to the Finsler-type norm \eqref{fin}.
We shall need the following lemma whose elementary proof is omitted; the dependence of the various functions on the paramaters
$\beta$ and $\gamma$ is here necessarily made explicit in the notation.
\begin{lemma}
Let $A_{\beta,\gamma}(\xi)$, $\xi\in\R^2$, be given by
\eqref{syl} and let $G_{\beta,\gamma}(\cdot,\cdot)$ and
$p_{\beta,\gamma}(\cdot)$ denote the associated parabolic Green function 
and Finsler-type norm respectively. Then the following symmetry relations
are valid:
\begin{align*}
\ia & \qquad G_{\beta,\gamma}(x_1 , -x_2 ,t) = G_{\beta,-\gamma}(x_1 , x_2 , t) \; , 
\qquad  x\in  \R^2 , \ t>0, \\
\ib & \qquad p_{\beta,\gamma}(x_1 , -x_2 ) = p_{\beta,-\gamma}(x_1 , x_2 ) \; , 
\qquad\qquad x\in\R^2.
\end{align*}
\la{lem:symmetry}
\end{lemma}
\proof Both parts follow directly from the relation
$A_{\beta,\gamma}(\xi_1 , -\xi_2) = A_{\beta,-\gamma}(\xi_1 , \xi_2)$, $\xi=(\xi_1,\xi_2)\in \R^2$.
\endproof

\subsection{Strong convexity}

The notion of strong convexity was introduced in \cite{EP1}
and applies to
a general elliptic symbol of order $2m$ acting on functions on
$\R^d$. We do not present here the general definition and simply
note that the symbol \eqref{syl} is
strongly convex if and only if the matrix
\be
M=
\left[
\begin{array}{ccc}
1  & \gamma  & \frac{\beta}{3} \\
\gamma & \frac{\beta}{3}  &  \gamma \\
\frac{\beta}{3}  &  \gamma & 1 
\end{array}
\right]
\la{g:M}
\ee
is positive definite. It is easily seen that this is the case
if and only if
\[
 0<\beta<3 \qquad  \mbox{ and } \qquad 
 \gamma^2 <\frac{\beta^2+3\beta}{18}.
\]
Our asymptotic formulas depend in a crucial way on where the point $(\beta,\gamma)\in E$ lies. For this purpose we define three open subsets of $E$:
\begin{align*}
E_0 \; = \; & \{ (\beta,\gamma) \in E \, :  \; 
 0<\beta<3 \quad  \mbox{ and } \quad \gamma^2 <\frac{\beta^2+3\beta}{18}
\} \\
E_1 \; = \; & \{ (\beta,\gamma) \in E \, :  \;  -1<\beta<3 \quad  \mbox{ and } \quad \gamma^2 > \frac{\beta^2+3\beta}{18} \} \\
E_2 \; = \; & \{ (\beta,\gamma) \in E \, :  \; \beta>3 \} 
\end{align*}
Defining also the curves
\begin{align*}
L_1 \; = \; & \{ (\beta,\gamma) \in E \, :  \;  0\leq \beta<3 \quad  \mbox{ and } \quad \gamma^2 = \frac{\beta^2+3\beta}{18} \} \\
L_2 \; = \; & \{ (\beta,\gamma) \in E \, :  \; \beta=3\},
\end{align*}
we then have the disjoint union
\[
E= E_0 \cup E_1 \cup E_2 \cup L_1 \cup L_2 \, .
\]


\begin{figure}[htbp]
\centering
\resizebox{\textwidth}{!}{%
\begin{tikzpicture}
\begin{axis}[
    axis on top,
    xlabel={\large $\beta$},
    ylabel={\large $\gamma$},
    axis lines=middle,
    xmin=-1.2, xmax=3.9,
    ymin=-1.3, ymax=1.3,
    xtick={},
    ytick={},
        extra x ticks={},
    extra x tick labels={\raisebox{-2pt}{\hspace{-4pt}\tiny0}},
    grid=none,
    grid style={line width=.1pt, draw=gray!30},
    major grid style={line width=.2pt, draw=gray!50},
    width=15cm, height=8.5cm, 
    tick label style={font=\tiny, fill=none, inner sep=1pt},
    every axis x label/.style={at={(current axis.right of origin)}, anchor=west, xshift=5pt},
    every axis y label/.style={at={(current axis.above origin)}, anchor=south, yshift=5pt},
    samples=200
]

    \addplot [name path=blue_top, domain=-1:3, forget plot, opacity=0] {(x+1)/4};
    \addplot [name path=blue_bottom, domain=-1:3, forget plot, opacity=0] {-(x+1)/4};

    \addplot [name path=green_top, domain=0:3, samples=150, forget plot, opacity=0] {sqrt((x^2+3*x)/18)};
    \addplot [name path=green_bottom, domain=0:3, samples=150, forget plot, opacity=0] {-sqrt((x^2+3*x)/18)};

    \addplot [name path=red_top, domain=3:3.9, samples=100, forget plot, opacity=0] {sqrt((x-1)/2)};
    \addplot [name path=red_bottom, domain=3:3.9, samples=100, forget plot, opacity=0] {-sqrt((x-1)/2)};

    \addplot [black!10] fill between [of=blue_top and blue_bottom, domain=-1:0];
    \addplot [black!10] fill between [of=blue_top and green_top, domain=0:3];
    \addplot [black!10] fill between [of=green_bottom and blue_bottom, domain=0:3];
    \addplot [black!45] fill between [of=green_top and green_bottom, domain=0:3];
    \addplot [black!25] fill between [of=red_top and red_bottom, domain=3:3.9];

    \addplot [domain=0:3, samples=150, thick, black!80!black] {sqrt((x^2+3*x)/18)};
    \addplot [domain=0:3, samples=150, thick, black!80!black] {-sqrt((x^2+3*x)/18)};
    \addplot [domain=-1:3, thick, black!80!black] {(x+1)/4};
    \addplot [domain=-1:3, thick, black!80!black] {-(x+1)/4};
    \addplot [domain=3:3.9, samples=100, thick, black!80!black] {sqrt((x-1)/2)};
    \addplot [domain=3:3.9, samples=100, thick, black!80!black] {-sqrt((x-1)/2)};
    \draw[thick, black] (axis cs:3,-1) -- (axis cs:3,1);

    \node[font=\footnotesize, color=black!80!black] at
    (axis cs:-0.17,0.1) {\large $E_1$};
    \node[font=\footnotesize, color=black!80!black!60!black] at
    (axis cs:1.8,0.2) {\large $E_0$};
    \node[font=\footnotesize, color=black!80!black] at (axis cs:3.45,0.4) {\large $E_2$};
    \node[font=\footnotesize, color=black!80!black] at (axis cs:1,0.36) {\large $L_1$};
    \node[font=\footnotesize, color=black!80!black] at (axis cs:3.12,0.15) {\large $L_2$};

    \node[circle, draw=black, fill=white, thick, inner sep=1.5pt, label={[font=\tiny, xshift=-12pt, yshift=1pt]above:{\large (3, 1)}}] at (axis cs:3,1) {};
    \node[circle, draw=black, fill=white, thick, inner sep=1.5pt, label={[font=\tiny, xshift=-14pt, yshift=-20pt]above:{\large (3, -1)}}] at (axis cs:3,-1) {};

\end{axis}
\end{tikzpicture}%
}
\caption{The sets $E_0, E_1, E_2, L_1, L_2$.}
\end{figure}
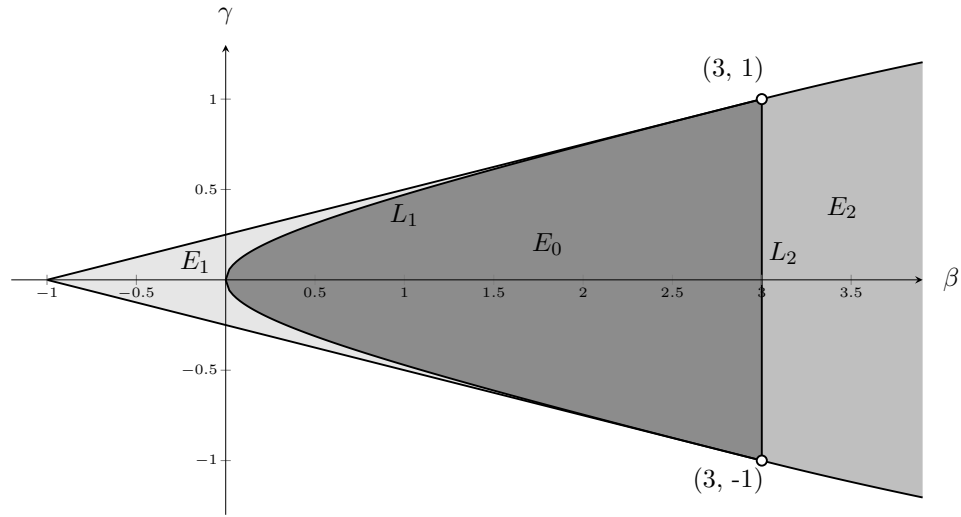


We shall present four theorems, each one corresponding to
$(\beta,\gamma)$ belonging to one of the sets $E_1$,
$E_2$, $L_1$ and $L_2$.
The strongly convex case $(\beta,\gamma)\in E_0$ is contained in
\cite{EP1} but the following simple lemma will be needed in order
to deal with the borderline cases $(\beta,\gamma)\in L_j$, $j=1,2$.
\begin{lemma}
\la{lem:sc}
Assume that $(\beta,\gamma)\in E_0$. Then
\[
 {\rm Re} \, A(\xi + i \eta )
 \geq -8 A(\eta), \quad \xi , \eta \in \mathbb{R}^2.
\]
Moreover the constant 8 is sharp and equality holds if and only if
$\xi=\pm \sqrt{3}\eta$.
\end{lemma}
\proof  Given $\xi,\eta\in\R^2$ let
\be
Y(\xi,\eta) = [\, \xi_1^2-3\eta_1^2 \,  , \;  
2(\xi_1\xi_2 -3\eta_1\eta_2)\,  , \;  \xi_2^2-3\eta_2^2 \, ]^{T}.
\la{g:Y}
\ee
A direct computation shows that
\be
{\rm Re}\,  A(\xi+i\eta) +8A(\eta) =Y(\xi,\eta)^T M \,  Y(\xi,\eta). 
\la{g:be}
\ee
Since $M$ is positive definite \eqref{g:be} follows and equality
holds if and only if $Y(\xi,\eta)=0$; equivalently, if and only if
$\xi=\pm \sqrt{3}\,\eta$.
\endproof


\section{The case $(\beta,\gamma) \in E_1$}
\la{sec:E1}


We recall that $(\beta,\gamma) \in E_1$ if and only if
\[
 -1 < \beta<3 \quad  \mbox{ and } \quad   \frac{\beta^2+3\beta}{18}
 < \gamma^2 < \frac{(\beta+1)^2}{16}.
\]

\begin{lemma}
\label{lem:E1}
Assume that $(\beta ,\gamma)\in E_1$. Then
\be
 {\rm Re} \, A(\xi +\textit{i} \eta )
 \geq -\kg A(\eta), \qquad \xi , \eta \in \mathbb{R}^2, 
\la{g:sy}
\ee
where
\[
\kg=\frac{8(1-\beta)+16\gamma^2}{(\beta +1)^2-16\gamma^2}
\]
Moreover the constant is sharp and equality holds 
precisely in the following cases:
\begin{align*}
& \hspace{2cm} \eta =(s,s), \qquad\;\; \xi = 
\pm \Big(\frac{3-\beta}{\beta+1-4\gamma}\Big)^{\nicefrac12}
(s ,-s )\\
& \hspace{1cm} \mbox{ or } \;\;\;\; \eta =(s,-s), \qquad
 \xi = \pm \Big(\frac{3-\beta}{\beta+1+4\gamma}\Big)^{\nicefrac12}
(s,s ) \\
\end{align*}
\end{lemma}
\proof
Given $\eta\in\R^2$ we define the symmetric matrix
$$
Q(\eta)=
\begin{bmatrix}
    1   & 2\gamma   & \hspace{-.5cm} \dfrac{\beta^2+5\beta-24\gamma^2}{3-\beta} &  -3\eo^2-\beta\et^2-6\gamma\eo \et \\[0.1cm]
  * & -\dfrac{4(\beta^2+\beta-12\gamma^2)}{3-\beta} & 2\gamma &
 -6\gamma(\eo^2+\et^2)\!-\! 4\beta\eo \et \\[0.1cm]
* & * & 1 &   -\beta\eo^2-3\et^2-6\gamma\eo \et \\[0.1cm]
    *  &  *  & *  &  \dfrac{(\beta-3)^2 A(\eta)}{(\beta+1)^2-16\gamma^2}
\end{bmatrix}$$
where  the symbol $*$ indicates definition by symmetry.
For $\xi\in\R^2$ we also set
\[
X(\xi)=[ \xo^2  , \;  \xo \xt , \;  \xt^2 , \; 1]^{T}.
\]
With simple computations we can see that
\begin{align} 
 {\rm Re} A(\xi +\textit{i} \eta )&+\kg A(\eta) \nonumber \\
  = \;\; & \xo^4+4\gamma \xo^3 \xt + 2\beta\xo^2 \xt^2 + 4\gamma \xo \xt^3+\xt^4 - 
  2\xo^2(3\eo^2+\beta\et^2+6\gamma\eo \et) \nonumber \\
 & -2\xt^2(\beta\eo^2+3\et^2+6\gamma\eo \et) - 4\xo\xt(3\gamma\eo^2+3\gamma\et^2+2\beta\eo \et) \label{g:ho} \\
& +  \dfrac{(3-\beta)^2}{(1+\beta)^2-16\gamma^2}A(\eta)  \nonumber \\
    = \;\;  &  X(\xi)^T Q(\eta)X(\xi) .  \nonumber
\end{align}
We consider the principal minors
$Q_k=Q_k(\eta)$, $k=1,2,3,4$, obtained from
$Q(\eta)$ starting from the top left corner, so that $Q_1=[1]$
and $Q_4=Q(\eta)$. The respective determinants are
\[
\hspace{-2.6cm}D_1 = 1 \;\; , \qquad\qquad D_2 = \dfrac{4(9\gamma^2-\beta^2+\beta\gamma^2-\beta)}{3-\beta}
\]
\[
D_3 = \dfrac{4(\beta^2+6\beta-3-24\gamma^2)(\beta^2+3\beta-18\gamma^2)\big( \,  (\beta+1)^2-16\gamma^2\big)}{(3-\beta)^3} \]
\[
\hspace{-1.9cm}D_4 = \dfrac{32(2\gamma^2-\beta+1)(\beta^2+3\beta-18\gamma^2)^2(\eo^2-\et^2)^2}{(3-\beta)^2}.
\]
\noindent It is straightforward to see that $D_1$, $D_2$ and $D_3$ are positive.
Moreover if $\eta_1^2 \neq \eta_2^2$ then $D_4>0$, hence
$X(\xi)^TQ(\eta)X(\xi) >0$ by Sylvester's criterion. It follows by continuity
that if $\eta_1^2 = \eta_2^2$ then $X(\xi)^TQ(\eta)X(\xi) \geq 0$.
Hence \eqref{g:sy} has been proved and it remains to
establish when  equality is valid.

Assume that equality holds in \eqref{g:sy}.
It follows by the preceding considerations that $\eta_1^2 = \eta_2^2$. Assume first that $\eta_1=\eta_2$ and, additionally and without loss of generality, that $\eta_1=\eta_2=1$. Then $D_4=0$ and $Q(\eta)$ has
a non-trivial kernel which is spanned by the vector
\[
v_1 =  [3-\beta \; , \; \beta - 3 \; , \; 3-\beta \; , \; \beta +1 -4\gamma]^T.
\]
Moreover the positive number $(24\gamma^2-\beta^2-6\beta+3)/(3-\beta)$
is an eigenvalue of $Q(\eta)$ with eigenvector $v_2 = [-1 , 0 , 1 , 0]^T$.
To consider the action of $Q$ on the remaining directions we
define the vectors
\[
v_3 = [1, \; 2, \; 1, \; 0]^T \;\; , \quad v_4= [4\gamma-\beta -1, \; \beta +1 -4\gamma, \; 4\gamma-\beta -1, \;3(3-\beta)]^T
\]
which are orthogonal to $v_1$ and $v_2$ and also to each other.
We then compute
\begin{align*}
& \left[
\begin{array}{cc}
Qv_3\cdot v_3 & Qv_4 \cdot v_3  \\
Qv_4\cdot v_3 & Qv_4 \cdot v_4
\end{array}
\right]  \\
& =
\left[
\begin{array}{cc}
2\, \frac{72\gamma^2-7\beta^2-4\beta+8(3-\beta)\gamma+3}{3-\beta}  & \frac{8\, (5\beta+18\gamma+3)\big( (\beta-\gamma-2)^2 +3(\gamma-1)^2   \big) }{\beta-3} \\[0.2cm]
\frac{8\, (5\beta+18\gamma+3)\big( (\beta-\gamma-2)^2 +3(\gamma-1)^2   \big)}
{\beta-3} & \frac{32 \, \big((\beta-\gamma-2)^2 +3(\gamma-1)^2 \big)}{\beta+1-4\gamma}
\end{array}
\right].
\end{align*}
Clearly $Qv_4 \cdot v_4>0$. Moreover the determinant of the above
matrix is equal to 
\[
 \frac{ 1152(\beta+1+4\gamma)(18\gamma^2-\beta^2-3\beta)
 \big((\beta-\gamma-2)^2 +3(\gamma-1)^2\big)^2}{(3-\beta)^2(\beta+1-4\gamma)} 
\]
and is also positive. Hence also $Qv_3 \cdot v_3>0$
and we conclude that $ X(\xi)^TQ(\eta)X(\xi)>0$
unless $X(\xi)$ is parallel to $v_1$, i.e. unless
\[
\xi_1^2 =\dfrac{3-\beta}{\beta+1-4\gamma}  \;, \qquad  \xi_1\xi_2 =-\dfrac{3-\beta}{\beta+1-4\gamma}
\;, \qquad \xi_1^2 =\dfrac{3-\beta}{\beta+1-4\gamma} ,
\]
or, equivalently,
$$
\xi = \pm \left( \Big(\dfrac{3-\beta}{\beta+1-4\gamma}
\Big)^{\nicefrac12} ,
-\Big(\dfrac{3-\beta}{\beta+1-4\gamma}\Big)^{\nicefrac12} \right)
$$
as required. 

Assume now that $\eta_1=-\eta_2$ and, without loss of generality, that $\eta_1=1$.
Expression \eqref{g:ho} for
${\rm Re} A(\xi +\textit{i} \eta )+\kg A(\eta)$
remains the same if we replace
$\xi_2,\eta_2$ and $\gamma$ by $-\xi_2,-\eta_2$ and $-\gamma$
respectively. Since the range of $\gamma$ under consideration
is symmetric, we conclude from the earlier discussion
that ${\rm Re} A(\xi +\textit{i} \eta )+\kg A(\eta)>0$ unless
$\xi_1^2=\xi_2^2=\xi_1\xi_2 =\frac{3-\beta}{\beta+1+4\gamma}$,
or, equivalently,
$\xi = \pm \Big( \big(\frac{3-\beta}{\beta+1+4\gamma}
\big)^{\nicefrac12} , \big(\frac{3-\beta}{\beta+1+4\gamma}
\big)^{\nicefrac12}\Big)$.
This concludes the proof. \endproof


For the next lemma we recall definition \eqref{fin}.
\begin{lemma}
\label{lem:3.2}
Assume that $(\beta,\gamma)\in E_1$.
The Finsler norms of the points $(1,\pm 1)$ are given by
\[
p(1,1) = \Big(\dfrac{8}{\beta+1+4\gamma}\Big)^{\nicefrac{1}{4}} \; , \qquad
p(1,-1) =\Big(\dfrac{8}{\beta+1-4\gamma}\Big)^{\nicefrac{1}{4}}.
\]
\begin{proof}
We shall prove that the above are true more generally for $-1<\beta<3$
and $|\gamma|<(\beta+1)/4$.
Let us denote $e_{t} =(\cos t , \sin t)$.
Writing a given $x\in\R^2$ as $x=|x|e_{\phi}$ we have
\[
p(x) = \max_{t\in [0,2\pi]} \frac{ x\cdot e_{t}}
{ A(e_{t})^{\nicefrac{1}{4}}} 
= |x| \max_{t\in [0,2\pi]} \frac{ \cos(t -\phi)}{ 
\Big(1 + 2\gamma \sin(2t)+ \dfrac{\beta -1}{2} \sin^2(2t)
\Big)^{\nicefrac{1}{4}}}.
\]
In particular
\[
p(1,1) =\sqrt{2} \max_{t\in [0,2\pi]} \frac{ \cos(t -
\frac{\pi}{4})}{ \Big(1 + 2\gamma \sin(2t)+ \dfrac{\beta -1}{2} \sin^2(2t)\Big)^{\nicefrac{1}{4}}}.
\]
To complete the proof it suffices to prove that the maximum is attained
at $t=\pi/4$. Equivalently, setting $\omega= t -\frac{\pi}{4}$,
we need to prove that
\[
\frac{\sqrt{2}\, \cos\omega}{\Big(
1+2\gamma \cos(2\omega) + \dfrac{\beta -1}{2} \cos^2(2\omega)
\Big)^{\nicefrac{1}{4}}}
\leq \Big(\dfrac{8}{\beta+1+4\gamma}\Big)^{\nicefrac{1}{4}} \; , \qquad 0\leq\omega\leq 2\pi .
\]
This will follow if we show that
\[
(\beta+1+4\gamma)\cos^4\omega \leq 
2+  4\gamma \cos(2\omega) +
(\beta -1) \cos^2(2\omega)   \; , \qquad 0\leq\omega\leq 2\pi .
\]
Setting $t=\cos\omega$ and using simple trigonomerty this is written
\[
1+\beta-4\gamma +(4-4\beta+8\gamma)t^2 +(3\beta-3-4\gamma)t^4\geq 0 \; , 
\qquad  t\in [-1,1].
\]
The last inequality is linear in $\gamma$ and is easily seen to be true
at the endpoints $\gamma= \pm (\beta+1)/4$;
hence the result for the point $(1,1)$ follows.
The result for $(1,-1)$
is now an immediate consequence of Lemma \ref{lem:symmetry} (ii).
\end{proof}
\end{lemma}

We shall now state the main result of this section which describes
the asymptotic behavior as $t\to 0+$ of the Green function $G(x,t)$ when 
the point $x\in\R^2\setminus\{(0,0)\}$  lies on one of the main diagonals.
For this we will apply the method of steepest descent to the
function $F(\lambda)$, cf. \eqref{fl}.

We fix a point $x\in\R^2$ and consider the function $\phi(z)$
defined in \eqref{PHI}. We note that
\begin{align}
& \phi_{z_1}(z)=ix_1-(z_1^3+3\gamma z_1^2z_2+\beta z_1z_2^2+\gamma z_2^3) ,
\nonumber \\
& \phi_{z_2}(z)=ix_2-(\gamma z_1^3+\beta z_1^2z_2+3\gamma z_1z_2^2+z_2^3)
\la{g:fcr1}
\end{align}
and also
\begin{align}
& \phi_{z_1z_1}(z)=-3z_1^2-6\gamma z_1z_2-\beta z_2^2 , \quad
\phi_{z_2z_2}(z)=-\beta z_1^2-6\gamma z_1z_2-3z_2^2  \nonumber \\
& \phi_{z_1z_2}(z)=-3\gamma z_1^2-2 z_1z_2-3\gamma z_2^2.
\la{g:fcr2}
\end{align}

\

We recall (cf. \eqref{sig2}) that $\sigma_2=3 
\cdot 2^{-\nicefrac{11}{3}}$, the 
exponential constant in the strongly convex case.
\begin{theorem}\label{thm:E1}
Assume that $(\beta,\gamma)\in E_1$ and let $\sigma_{\beta,\gamma} = 
\Big(\frac{(\beta+1)^2-16\gamma^2}{2\gamma^2-\beta +1}\Big)^{\nicefrac{1}{3}}\sigma_2$.
For a point $x\in\R^2$ of the form $x=(s,\pm s)$, $s\neq 0$,
the following asymptotic formula is valid as $t \rightarrow 0+$:
\be
G(x,t) =
\frac{1}{4\pi \sqrt{3-\beta}\sqrt{\sigma_{\beta ,\gamma}}} \cdot 
\frac{1}{p(x)^{\nicefrac{2}{3}} \, t^{\nicefrac{1}{3}}}
\exp\Big\lbrace- \sigma_{\beta,\gamma} \dfrac{p(x)^{\nicefrac{4}{3}}}{t^{\nicefrac{1}{3}}} \Big\rbrace \big(  1+O(t^{\nicefrac13}) \big).
\la{g:asym}
\ee
\begin{proof} 
Recalling \eqref{g:GF} and \eqref{fl}, we need to consider
the asymptotic behavior as $\lambda\to +\infty$  of the function
\[
F(\lambda)=\int\limits_{\mathbb{R}^2} e^{\lambda \phi(z)}\,dz \, .
\]
{\em Case 1.} $x=(s,s)$.  Due to the scaling relation
\[
G(sy,t)= \frac{1}{s^2}G(y,\frac{t}{s^4})  \; , \qquad y\in\R^2, \;\;
s\in\R , \;\; t>0,
\]
it is enough to consider the case $x=(1,1)$. Hence the phase function
is
\[
\phi(z) = i(z_1+z_2) -\frac14
\big(z_1^4+4 \gamma z_1^3z_2+2\beta z_1^2z_2^2+4 \gamma z_1z_2^3+z_2^4\big).
\]
Let
\[
\xi_0=\frac{1}{2}\Big(\frac{\beta + 1-4\gamma}{2\gamma^2-\beta+1}\Big)^{\nicefrac{1}{3}}
\Big(\frac{3-\beta}{\beta +1-4\gamma}\Big)^{\nicefrac12}(1,-1), \qquad
\eta_0 = \frac{1}{2}\Big(\frac{\beta + 1-4\gamma}{2\gamma^2-\beta+1}\Big)^{\nicefrac{1}{3}}(1,1).
\]
Using \eqref{g:fcr1} it is straightforward to see that the points
\[
z_0^{\pm}=\pm\xi_0+i\eta_0
\]
are critical points of $\phi$. We shall prove that these are the points
that contribute to the asymptotic behavior of $F(\lambda)$.

By Cauchy's theorem we have
\be
F(\lambda)=\int\limits_{\mathbb{R}^2} e^{\lambda\phi(z)}\,dz = \int\limits_{\R^2 +i\eta_0} e^{\lambda\phi(z)}\,dz, \qquad \lambda>0.
\la{g:cauchy}
\ee
Moreover setting 
$S(R) = \{ z\in \R^2 +i\eta_0 : |z_1|^2 +|z_2|^2 =R^2\}$
we have
\be
\lim_{R\to +\infty} \max\{ {\rm Re}\,\phi(z) : \; z\in S(R)\} =-\infty
\la{g:cond2}
\ee
Now, it is easily seen that $\phi(z_0^+)=\phi(z_0^-)$.
We claim that
\be
{\rm Re}\, \phi(z) \leq {\rm Re}\,\phi(z_0^\pm), \qquad z\in \mathbb{R}^2+i\eta_0,
\la{g:cond3}
\ee
with equality precisely at the points $z_0^\pm$. 
The last inequality may be written equivalently as
\[
{\rm Re}\, A(\xi +i \eta_0) \geq {\rm Re}\,  
A(\pm \xi_0 +i\eta_0) \; , 
\quad \xi\in \R^2.
\]
Now, the points $(\pm\xi_0 ,\eta_0)$  satisfy the conditions
of Lemma \ref{lem:E1}, therefore ${\rm Re}\,  A(z_0^\pm)
=-\kbg A(\eta_0)$. Hence the required inequality follows from
Lemma \ref{lem:E1} and equality holds if and only $\xi=\pm \xi_0$.

So \eqref{g:cond3} has been proved. This together with equations \eqref{g:cauchy} and \eqref{g:cond2} are the conditions under which
\cite[Criterion 1, page 15]{EP1} can be applied;
see also \cite[Section 4.5]{F2}. The conclusion is
that the asymptotic behavior of $F(\lambda)$ as $\lambda\to +\infty$
is given as the sum of two contributions, one from $z_0^+$ and one from $z_0^-$. 
More precisely, we have
\[
F(\lambda) =   {\rm contr}(z_0^+)
\big(1+O(\frac{1}{\lambda} \, )\big) +   {\rm contr}(z_0^-)   
\big(1+O(\frac{1}{\lambda} \, )\big) 
\]
where the two contributions are,
cf. \cite[equation (3.6)]{EP1} or \cite[equation (1.61)]{F2},
\[
{\rm contr}(z_0^{\pm}) =  \frac{2\pi}{\lambda} \big( {\rm det} (\phi_{z_iz_j})\big|_{z=z_0^{\pm}} \big) ^{-1/2}
e^{\lambda\phi(z_0^{\pm})}.
\]
We have
\[
\phi(z_0^+) = \phi(z_0^-) = -\dfrac{3}{4}
\Big(\frac{\beta + 1-4\gamma}{2\gamma^2-\beta+1}\Big)^{\nicefrac{1}{3}}.
\]
Moreover, using \eqref{g:fcr2},
\[
\det(\phi_{z_iz_j})\mid_{z=z_0^+} = \det(\phi_{z_iz_j})\mid_{z=z_0^-}=3(3-\beta)\Big(\frac{\beta + 1-4\gamma}{2\gamma^2-\beta+1}\Big)^{\nicefrac{1}{3}}.
\]
Adding the two contributions and recalling from Lemma \ref{lem:3.2} that
\[
(\beta+1-4\gamma)^{\nicefrac{1}{3}} =2^{-1}\left[(\beta+1)^2-16\gamma^2\right]^{\nicefrac{1}{3}}p(1,1)^{\nicefrac{4}{3}}
\]
we obtain
\begin{eqnarray}
 &  & \hspace{-1.5cm} {\rm contr}(z_0^+)+ {\rm contr}(z_0^-)
\nonumber \\
& = & \frac{2\pi}{\lambda}\left[ \det(\phi_{z_iz_j})\mid_{z=z_0^+}\right]^{-\nicefrac{1}{2}} \exp\lbrace \lambda \phi(z_0^+)\rbrace 
\nonumber \\
&  & + \frac{2\pi}{\lambda}\left[ \det(\phi_{z_iz_j})\mid_{z=z_0^-}\right]^{-\nicefrac{1}{2}} \exp\lbrace \lambda \phi(z_0^-)\rbrace 
\nonumber \\
& =& \frac{4\pi}{\sqrt{3(3-\beta)}\lambda} 
\Big( \frac{2\gamma^2-\beta+1}{\beta + 1-4\gamma}\Big)^{\nicefrac{1}{6}} 
\exp \left\lbrace -\frac{3}{4}\Big( \frac{\beta + 1-4\gamma}{2\gamma^2-\beta+1}\Big)^{\nicefrac{1}{3}} \lambda \right\rbrace  \nonumber \\
&=&  \frac{2^{\nicefrac{5}{2}}\pi}{\sqrt{3(3-\beta)} \, \lambda}\left( \frac{2\gamma^2-\beta+1}{(\beta+1)^2-16\gamma^2}\right)^{\nicefrac{1}{6}}
p(1,1)^{-\nicefrac{2}{3}} \nonumber \\
& &\times \exp \left\lbrace - \dfrac{3}{8}
\left( \frac{(\beta+1)^2-16\gamma^2}{2\gamma^2-\beta+1}\right)^{\nicefrac{1}{3}} p(1,1)^{\nicefrac{4}{3}}\lambda \right\rbrace \la{g:stef}
\end{eqnarray}
Recalling \eqref{g:GF} we thus arrive at
\begin{align*}
G(x,t)  = & \frac{1}{4\pi\sqrt{3-\beta}} 
\Big( \frac{2\gamma^2-\beta+1}{(\beta + 1)^2-16\gamma^2}\Big)^{\nicefrac{1}{6}} p(1,1)^{-\nicefrac{2}{3}} t^{-\nicefrac{1}{3}} \\
& \qquad \times
\exp \bigg\lbrace -\frac{3}{2^{\nicefrac{11}{3}}}\left( \frac{(\beta + 1)^2-16\gamma^2}{2\gamma^2-\beta+1}\right)^{\nicefrac{1}{3}}\frac{ p(1,1)^{\nicefrac{4}{3}}}{t^{\nicefrac{1}{3}}} \bigg\rbrace \big(1+O(t^{\nicefrac13} \, )\big) 
\end{align*}
which is \eqref{g:asym} for $x=(1,1)$.

\smallskip

\noindent
{\em Case 2.} $x=(s,-s)$. We will use the symmetry of the symbol
to obtain the result from the previous case. In particular
we will use indices to indicate the dependence on $\beta$
and $\gamma$ as in Lemma \ref{lem:symmetry}.
Let us define the function
\[
F_{\beta,\gamma}(p,t)= \frac{1}{4\pi \sqrt{3-\beta}\sqrt{\sigma_{\beta ,\gamma}}}
\cdot 
\frac{1}{p^{\nicefrac{2}{3}} t^{\nicefrac{1}{3}}}
\exp\Big\lbrace - \sigma_{\beta,\gamma} \dfrac{p^{\nicefrac{4}{3}}}{t^{\nicefrac{1}{3}}} \Big\rbrace \; , \qquad p, \, t>0,
\]
so that \eqref{g:asym} can be written, in an self-evident notation, as
\[
G(x,t) \sim F_{\beta,\gamma} \big( p_{\beta,\gamma}(x) ,t \big)  \; , \qquad
\mbox{ as } t\to 0+.
\]
Noting that $F_{\beta,\gamma}=F_{\beta,-\gamma}$ and using Lemma \ref{lem:symmetry} 
we have 
\[
G_{\beta,\gamma}(s,-s,t)=   G_{\beta,-\gamma}(s,s,t) \sim
F_{\beta,-\gamma} (p_{\beta,-\gamma}(s ,s) ,t) 
=  F_{\beta,\gamma} (p_{\beta,\gamma}(s ,-s) ,t),
\]
as required. This completes the proof.
\end{proof}
\end{theorem}

\section{The case $(\beta,\gamma) \in L_1$}

This borderline case is different to the previous 
one as now there are four relevant saddle points for each of the 
values of $x$ that we consider. This leads to a qualitatively
different asymptotic behavior of the Green function.

We recall that $(\beta,\gamma) \in L_1$ if and only if
\[
 -1 < \beta<3 \quad  \mbox{ and } \quad \gamma^2 =
 \frac{\beta^2+3\beta}{18}.
\]

\begin{lemma}\label{lem:L1}
Assume that $(\beta ,\gamma)\in L_1$. Then
\be
 {\rm Re} \, A(\xi + i \eta )
 \geq - 8 A(\eta), \qquad \xi , \eta \in \mathbb{R}^2.
\la{g:sy1}
\ee
Moreover the constant is sharp and (i) in case $(\beta,\gamma)
\neq (0,0)$ equality holds if and only if
\begin{align*}
& \hspace{2cm} \eta =(s,s), \qquad \; \xi = 
\pm \Big(\frac{3-\beta}{\beta+1-4\gamma}\Big)^{\nicefrac12} ( s \,
, \; -s )\\
& \hspace{1.2cm} \mbox{ or } \;\; \eta =(s,-s), \qquad
 \xi = \pm \Big(\frac{3-\beta}{\beta+1+4\gamma}\Big)^{\nicefrac12}
 ( s \, , \; s)  \\
& \hspace{1.2cm} \mbox{ or } \;\;
\gamma (\xi_1^2-3\eta_1^2) =\gamma (\xi_2^2-3\eta_2^2) 
=-\frac{\beta}{3} (\xi_1\xi_2-3\eta_1\eta_2)
\end{align*}
while (ii) in case $(\beta,\gamma)= (0,0)$ equality holds if and only if
\[
\xi_1^2-3\eta_1^2 =\xi_2^2-3\eta_2^2 =0.
\]
\end{lemma}
\proof (i) Since $L_1=\overline{E_1}\cap \overline{E_0}$ (closures in the topology 
of  $E$), inequality \eqref{g:sy1} follows directly from
Lemma \ref{lem:sc}.
Moreover, from Lemmas \ref{lem:sc} and \ref{lem:E1} follows that if
any of the conditions in the statement is satisfied then equality is
valid in \eqref{g:sy1}.

We shall prove that there exist no other cases where equality holds.
So assume that we have equality for some $\xi,\eta\in\R^2$.
Relation \eqref{g:be} (with $M$ being given by \eqref{g:M})
is still valid. Now, however, $M$ has zero as a simple
eigenvalue. Hence if ${\rm Re}\,  A(\xi+i\eta) +8A(\eta)=0$
and $\xi\neq \pm \sqrt{3}\,\eta$, then necessarily
the vector $Y(\xi,\eta)$ defined in \eqref{g:Y} must be a non-zero
element of the kernel of $M$. The latter is spanned by
the vector $[ \beta , \; -6\gamma , \; \beta]^{T}$, hence the
third of the three relations above is valid.

The proof in case (ii) is similar and simpler and is omitted.
\endproof

\begin{theorem}
Assume that $(\beta,\gamma)\in L_1$. If the point
$x\in\R^2$ is 
of the form $x=(s,\pm s)$, $s\neq 0$, then the following asymptotic formula  is valid as $t \rightarrow 0^+$:
\begin{align*}
G(x,t)\; = \; &  \dfrac{1}{4\pi\sqrt{3-\beta}\sqrt{\sigma_2}}
\cdot 
\frac{1}{p(x)^{\nicefrac{2}{3}} \, t^{\nicefrac{1}{3}}}
\exp\Big\lbrace - \sigma_2 \dfrac{p(x)^{\nicefrac{4}{3}}}{t^{\nicefrac{1}{3}}}
\Big\rbrace  \\
& \qquad\qquad \times \; \left( 1+\cos 
\Big( \sqrt{3}\,\sigma_2 \dfrac{p(x)^{\nicefrac{4}{3}}}{t^{\nicefrac{1}{3}}}-\frac{\pi}{3}\Big)  + O(t^{\nicefrac13})\right).
\end{align*}
\begin{proof}

\medskip

We shall only consider the case $(\beta,\gamma)\neq (0,0)$; the
case $(\beta,\gamma)= (0,0)$ is much simpler.

{\em 1. Case $x=(s,s)$.} As in the proof of Theorem \ref{thm:E1},
we assume that $x=(1,1)$.
By the proof of Theorem \ref{thm:E1} and continuity we obtain that
the points $z_0^{\pm}=\pm\xi_0+i\eta_0$ where
\[
\xi_0=\frac{1}{2}\Big(\frac{\beta + 1-4\gamma}{2\gamma^2-\beta+1}\Big)^{\nicefrac{1}{3}}
\Big(\frac{3-\beta}{\beta +1-4\gamma}\Big)^{\nicefrac12}(1,-1), \qquad
\eta_0 = \frac{1}{2}\Big(\frac{\beta + 1-4\gamma}{2\gamma^2-\beta+1}\Big)^{\nicefrac{1}{3}}(1,1).
\]
are saddle points of $\phi$. It is easily seen that there are two
more saddle points on the plane $\R^2+ i\eta_0$, namely
$z_*^{\pm}=\pm\xi_*+i\eta_0$ where $\xi_*=\sqrt{3}\, \eta_0$.

Arguing as in the  proof of Theorem \ref{thm:E1}, in order to show that the 
points $z_0^\pm, \, z_*^\pm $ are precisely those that contribute to the 
asymptotic behavior, we must establish that
\[
{\rm Re}\, \phi(z) \leq {\rm Re}\, \phi(z_0^{\pm}) =
{\rm Re}\, \phi(z_*^{\pm}) \; , \quad z\in\R^2+i\eta_0,
\]
with equality exactly at the points $z=z_0^\pm, \, z_*^\pm $.
Equivalently, we must show that
\be
{\rm Re}\, A(\xi+i\eta_0) \geq {\rm Re}\, A(\xi_0^{\pm}+i\eta_0) =
{\rm Re}\, A(\xi_*^{\pm}+i\eta_0) \; , \quad \xi\in\R^2,
\la{gs}
\ee
with equality precisely at the points $\xi_0^\pm, \, \xi_*^\pm$.
For this we note that the
four points $(\pm\xi_0 ,\eta_0)$, $(\pm \xi_*,\eta_0)$ satisfy the 
conditions of Lemma \ref{lem:L1}, therefore
\[
 {\rm Re}\, A(\xi_0^{\pm}+i\eta_0) =
{\rm Re}\, A(\xi_*^{\pm}+i\eta_0) =- 8A(\eta_0).
\]
It now follows from Lemma \ref{lem:L1} that \eqref{gs} is
valid and that it is precisely at the points
$\xi=\xi_0^\pm, \, \xi_*^\pm$ that equality holds.

The sum of the contributions of $z_0^+$ and $z_0^-$ is obtained
by taking the limit $\gamma\to \sqrt{(\beta^2+3\beta)/18}$ in \eqref{g:stef}, so
\[
{\rm contr}(z_0^+) + {\rm contr}(z_0^-) =
\frac{2^{\nicefrac{5}{2}} \pi}{\sqrt{3(3-\beta)} \, \lambda} \, p(1,1)^{-\nicefrac{2}{3}}
\exp \Big\lbrace - \frac38 \, p(1,1)^{\nicefrac{4}{3}}\lambda\Big\rbrace.
\]
We also compute
\[
\phi(z_*^{\pm}) = -\dfrac{3}{2}
\bigg(\frac{9(\beta+1)-6\sqrt{2}\sqrt{\beta^2+3\beta}}{(\beta-3)^2}\bigg)^{\nicefrac{1}{3}}
e^{ \mp\nicefrac{\pi i}{3}} 
\]
and
\[
\det(\phi_{z_iz_j})\mid_{z=z_*^{\pm}} =
 3(3-\beta)\bigg(\frac{9(\beta+1)-6\sqrt{2}\sqrt{\beta^2+3\beta}}{(\beta-3)^2}\bigg)^{\nicefrac{1}{3}} 
 e^{\pm\nicefrac{2\pi i}{3}}.
\]
Hence
\begin{align*}
& {\rm contr}(z_*^+) + {\rm contr}(z_*^-) \\
 &= \frac{2\pi}{\lambda}
\Big(\det(\phi_{z_iz_j})\mid_{z=z_*^{+}}\Big)^{-\nicefrac{1}{2}}
e^{\lambda \phi(z_*^+)}  +\frac{2\pi}{\lambda}
\Big(\det(\phi_{z_iz_j})\mid_{z=z_*^{-}}\Big)^{-\nicefrac{1}{2}}
e^{\lambda \phi(z_*^-)} \\
&= \frac{4\pi}{\sqrt{6}\lambda}
\bigg(\frac{9(\beta+1)-6\sqrt{2}\sqrt{\beta^2+3\beta}}{(\beta-3)^2}
\bigg)^{-\nicefrac{1}{6}}\!\!\!\!\!\!\!\!
\exp\Big\{\! -\frac34 
\Big(\frac{9(\beta+1)-6\sqrt{2}\sqrt{\beta^2+3\beta}}{(\beta-3)^2}\Big)^{\nicefrac{1}{3}}\!\!\!\!\lambda\Big\} \\
&\times \cos\bigg[ \frac{3\sqrt{3}}{4}\Big(\frac{9(\beta+1)-6\sqrt{2}\sqrt{\beta^2+3\beta}}{(\beta-3)^2}\Big)^{\nicefrac{1}{3}}
\lambda -\frac{\pi}{3}\Big].
\end{align*}
By Lemma \ref{lem:3.2} and continuity we also have
\[
\bigg(\frac{9(\beta+1)-6\sqrt{2}\sqrt{\beta^2+3\beta}}{(\beta-3)^2}\bigg)^{\nicefrac{1}{3}}  =\frac12 \; p(1,1)^{\nicefrac{4}{3}}
\]
and we therefore obtain
\begin{align*}
& {\rm contr}(z_*^+)  + {\rm contr}(z_*^-) \\
&= \frac{2^{\nicefrac{5}{2}} \pi}{ \sqrt{3(3-\beta)} \, \lambda} p(1,1)^{-\nicefrac{2}{3}}
\exp \Big\lbrace - \frac38 \, p(1,1)^{\nicefrac{4}{3}}\lambda\Big\rbrace
\cos\Big( \frac{3\sqrt{3}}{8}p(1,1)^{\nicefrac{4}{3}}
\lambda -\frac{\pi}{3}\Big).
\end{align*}
Hence we arrive at
\begin{align*}
 F(\lambda) = & \; \frac{2^{\nicefrac{5}{2}} \pi}{\sqrt{3(3-\beta)} } p(1,1)^{-\nicefrac{2}{3}}\\
& \times  \frac{1}{\lambda}
\exp \Big\lbrace - \frac38 \, p(1,1)^{\nicefrac{4}{3}}\lambda\Big\rbrace 
 \bigg( \! 1 + \cos\Big( \frac{3\sqrt{3}}{8}\, 
p(1,1)^{\nicefrac{4}{3}} \lambda - \frac{\pi}{3} \Big) +O(\frac{1}{\lambda})\bigg).
\end{align*}
Applying \eqref{g:GF} concludes the proof.

\medskip

\noindent
{\em 2. Case $x=(s,-s)$.} This follows from the previous case
as in the proof of Theorem \ref{thm:E1}.
\end{proof}
\end{theorem}


\section{The case $(\beta,\gamma) \in E_2$}
\la{sec:E2}


We recall that $(\beta,\gamma) \in E_2$ if and only if
\be
\beta > 3 \quad  \mbox{ and } \quad   \gamma^2 <\frac{\beta-1}{2}.
\la{g:assum}
\ee
We set
\be
q=\sqrt{ (\beta+1)^2 -16\gamma^2} \; , \qquad\qquad \mu=\frac{ q+ \beta+1 }{4} \, ,
\vspace{0.2cm}
\la{q}
\ee
constants that will appear throughout what follows. We note that by \eqref{g:assum}
\[
(\beta+1)^2 -16\gamma^2 > (\beta+1)^2 -16 \cdot \frac{\beta-1}{2}
=(\beta -3)^2
\]
and therefore
\[
q - \beta +3 >0  .
\]
We also note that
\be
\mu^2-\gamma^2= \frac{ q(q+\beta+1)}{8}>0.
\la{g:beli}
\ee
We shall need the following simple lemma whose proof is omitted.
\begin{lemma}
Let $\eta =(\cos\theta ,\sin\theta)$
and $\xi=r \, (\sin\theta,\cos\theta)$.
Then
\begin{align*}
 A(\xi+ i\eta) =& \; \, (r^4+1)\Big(  1
+2\gamma \sin(2\theta) + \frac{\beta-1}{2}\sin^2(2\theta)\Big) \\
&\qquad - r^2\Big( 2\beta +12\gamma\sin(2\theta)
+ (\beta+3)\sin(2\theta)^2  \Big) \\
& +i \, \bigg\{
2(r^3-r)\Big( 2\gamma +(\beta+1)\sin(2\theta) + 2\gamma\sin^2(2\theta)
\Big)\bigg\}.
\end{align*}
\la{g:agg3}
\end{lemma}
We define
\[
\kbg = \frac{ 8 \, q \, (q+\beta-3)}{ (q-\beta+3)^2},
\]

\begin{lemma}
\la{lem:E2}
Assume that $(\beta ,\gamma)\in E_2$.
There holds
\be
\la{g:pol3}
{\rm Re}\, A(\xi +i\eta) \geq - \kbg A(\eta)  \; , \qquad
\xi, \, \eta\in\R^2.
\ee
Moreover the constant $\kbg$ is the best possible and equality 
holds  if and only if
\begin{align}
\ia  &  \mbox{ the point $\eta\in\R^2$ satisfies $
\gamma\eta_1^2 +2\mu\eta_1\eta_2 +\gamma\eta_2^2 =0$} 
\nonumber \\
 \mbox{ and }\ib & \mbox{ the point $\xi$ is one of 
the two points $\pm (r\eta_2 ,r\eta_1 )$ where $r>0$
is defined by}\nonumber  \\[0.2cm]
& \hspace{2.2cm}
r^2 =\frac{3q+\beta-3}{q-\beta+3}.
\la{g:bo1}
\end{align}
\end{lemma}
{\em Remark.} We note that such points $(\xi,\eta)$ do exist since
the discriminant of the above quadratic polynomial in (i)
is positive by \eqref{g:beli}.
\proof
We set for simplicity
$T(\xi,\eta) ={\rm Re}\, A(\xi +i\eta) +\kbg A(\eta)$
and so we must prove that
\[
T(\xi,\eta) \geq 0 \; , \qquad
\xi, \; \eta\in\R^2.
\]
with equality if and only if (i) and (ii) above hold true.

Given $\eta\in\R^2$ we define the symmetric matrix
\[
Q(\eta)=\begin{bmatrix}
 1 & 2\gamma   &   \frac{\beta^2-5\beta+6-(\beta-6)q}
 {3q+\beta-3} & -(3\eo^2+6\gamma\eo\et +\beta\eta_2^2) 
  \\[0.1cm]
 *  &  \frac{4\beta-12 +(8\beta-12)q}{3q+\beta-3} & 2\gamma  &
 -(4\beta \eta_1\eta_2 +6\gamma |\eta|^2) \\[0.1cm]
* & *  & 1  &  -(\beta\eta_1^2 +6\gamma\eo \et 
+3\eta_2^2) \\[0.1cm]
*   &  *  & *  & (\kbg +1) A(\eta)
\end{bmatrix}
\]
where  the symbol $*$ indicates definition by symmetry. We also set
$$
X(\xi)=[ \xo^2  , \;   \xo \xt , \;   \xt^2 , \; 1 ]^{T} \; , 
\qquad \xi\in\R^2.
$$
It is a simple computation to verify that
\[
T(\xi,\eta)= X(\xi)^T Q(\eta)X(\xi) \; , \qquad
\xi, \; \eta\in\R^2.
\]
We consider the principal minors $Q_i$, $i=1,2,3,4$ of $Q$ (so $Q_1=[1]$
and $Q_4=Q$) and the correspodning determinants $D_i=\det (Q_i)$.
A  direct computation gives
\begin{eqnarray*}
D_2 &=& \frac{4\beta-12 +(8\beta-12)q}{\beta-3+3q} 
-4\gamma^2 \\
&>& \frac{4\beta-12 +(8\beta-12)q}{\beta-3+3q} -2(\beta-1) \\
&=& \frac{2(\beta-2)(q -\beta+3)}{3q+\beta-3} \\
&>& 0.
\end{eqnarray*}
We also find
\[
D_3 = \frac{ (\beta-3)(3q -\beta+3)(q-\beta+3)^2
 \Big(\beta^2 +6\beta -3-24\gamma^2 
+3(\beta-3)q \Big)}{ (q+\beta-3)^3}
\]
and therefore $D_3>0$ since
\[
\beta^2 +6\beta -3-24\gamma^2 > \beta^2 +6\beta -3 -12(\beta-1)
=(\beta-3)^2>0.
\]
Concerning $D_4=\det(Q)$, it is a matter of lengthy but otherwise 
simple computations to see that
\[
D_4  = 16 \, (\beta-3)^2 \; \frac{q+ \beta-3}{q+\beta+1}
  \bigg(\frac{3q-\beta+3}{3q+\beta-3}\bigg)^2 
\big(\gamma \, \eta_1^2 + 2\mu \, \eta_1\eta_2
 +\gamma \, \eta_2^2\big)^2.
\]

It follows that if 
$\gamma\, \eta_1^2 + 2\mu \,  \eta_1\eta_2 +\gamma\, 
\eta_2^2 \neq 0$
then $Q(\eta)$ is positive definite and therefore 
$T(\xi,\eta)>0$ for all $\xi\in\R^2$.

Assume now that
\[
\gamma \eta_1^2 + 2\mu   \eta_1\eta_2 +\gamma\eta_2^2 =0.
\]
We then have $T(\xi,\eta)\geq 0$ for all $\xi, \,\eta\in\R^2$ by
continuity, so it remains to prove that equality holds precisely if
$\xi\in\R^2$ is one of the two points defined by (ii).
So let 
\[
(\xi_1,\xi_2)= r \,  (\eta_2,\eta_1).
\]
It is a matter of direct computation to verify
that $Q(\eta)X(\xi)=0$; see Lemma \ref{lem:app} in the Appendix.
So $T(\xi,\eta)=0$ as required. Since $X(-\xi)
=X(\xi)$, we also have $T(-\xi,\eta)=0$.

Finally, we must show that no other such points $\xi$ exist.
Indeed, assume that some $\xi' \in\R^2$ satisfies
$T(\xi' , \eta)=0$.
Then $X(\xi')$ belongs in the kernel of $Q(\eta)$. Since $D_1, \, D_2, \, D_3>0$, the kernel is one-dimensional
and hence $X(\xi')$ is a scalar multiple of $X(\xi)$. We conclude
that $\xi'$ is indeed one of the points $\pm\xi$.
\endproof

\begin{lemma}
\la{lem:5.3}
Let $(\beta ,\gamma)\in E_2$.
Assume that $\eta=(\cos\theta, \sin\theta)$ satisfies condition (i) of Lemma
\ref{lem:E2}. Then
\begin{eqnarray*}
\ia &&   \gamma +\frac{\beta+1}{2}\sin(2\theta) 
+ \gamma\sin^2(2\theta) =0 \, , 
\\
\ib &&   A(\eta)=  \frac{q (q-\beta+3)}{2(q+\beta+1)}.
\end{eqnarray*}
\end{lemma}
\proof We have
\[
\gamma +\mu\sin(2\theta) = \gamma\cos^2\theta +2\mu\cos\theta\sin\theta +\gamma\sin^2\theta
= 0  \, ,
\]
hence
\be
\sin(2\theta) = -\frac{\gamma}{\mu} = -\frac{4\gamma}{ q+\beta+1}
\la{g:sif}
\ee
Now (i) and (ii) follow after substitution, noting also for (ii) that
\[
A(\eta)  =  1 +2\gamma \sin(2\theta) + \frac{\beta-1}{2}\sin^2(2\theta).
\qedhere
\]

\begin{lemma}
\la{lem:5.4}
Let $(\beta ,\gamma)\in E_2$ and
let $(\xi,\eta)\in\R^4$ be a point for which equality is valid in
\eqref{g:pol3}. Writing $\eta=|\eta|(\cos\theta , \sin\theta)$,
at this point $(\xi,\eta)$ we have
\begin{eqnarray*}
\ia  &&  \big( {\rm Re}\, A(\xi +i\eta) \big)_{\eta_1}
=- \frac{ 8 \, q \, (q+\beta-3)}{ q-\beta +3} \,  |\eta|^2 \eta_1  \cos(2\theta) \, , \\
\ib &&  \big( {\rm Re}\, A(\xi +i\eta) \big)_{\eta_2}
=\frac{ 8 \, q \, (q+\beta-3)}{ q-\beta +3} \,  |\eta|^2 \eta_2  \cos(2\theta) .
\end{eqnarray*}
\end{lemma}
\proof (i) We first note that at any point $(\xi,\eta)$ we
have
\begin{eqnarray*}
\big( {\rm Re}\, A(\xi +i\eta) \big)_{\eta_1} &=& -12(\eta_1 +
\gamma\eta_2)\xi_1^2 -4(\beta\eta_1+3\gamma\eta_2)\xi_2^2 -8(3\gamma\eta_1+\beta\eta_2)\xi_1\xi_2   \\
&&+4\eta_1^3 +12\gamma\eta_1^2\eta_2 +4\beta\eta_1\eta_2^2
 +4\gamma\eta_2^3 .
\end{eqnarray*}
Without loss of generality we assume that $|\eta|=1$. We also assume that
$\eta_1\neq 0$, the case $\eta_1=0$ can be obtained 
separately or by a limiting argument.

Taking $\eta=(\cos\theta , \sin\theta)$ and $\xi=(r\sin\theta,r\cos\theta)$
we find
\begin{eqnarray*}
&& \hspace{-1.5cm}
\frac{\big( {\rm Re}\, A(\xi +i\eta) \big)_{\eta_1}}{\eta_1} \\
&=& 4(1-\beta r^2)\cos^2\theta +4(\beta-3r^2 -2\beta r^2)
\sin^2\theta  \\
&& +12\gamma (1-3r^2)\cos\theta \sin\theta 
+4\gamma (1-3r^2) \frac{\sin^3\theta}{\cos\theta} \\
&=& 2(\beta+1)(1-3r^2) +2\big(1-\beta +(\beta+3)r^2\big)\cos(2\theta) \\
&& +6\gamma (1-3r^2)\sin(2\theta) + 2\gamma (1-3r^2)\, 
\frac{\cos^2(2\theta)-2\cos(2\theta)+1}{\sin(2\theta)} \\
&=& (1-3r^2)\Big( 2(\beta+1)+4\gamma\sin(2\theta) +\frac{4\gamma}{\sin(2\theta)}\Big) \\
&& +  \Big(  2 \big( 1-\beta +(\beta+3)r^2 \big) -\frac{4\gamma(1-3r^2)}
{\sin(2\theta)}\Big)\cos(2\theta).
\end{eqnarray*}
The first term in the RHS vanishes by (i) of Lemma \ref{lem:5.3}.
For the second term we
substitute for $r^2$ and $\sin(2\theta)$ using \eqref{g:bo1} and 
\eqref{g:sif} respectively  and (i) follows.
To prove (ii) we interchange the roles of $\eta_1$ and $\eta_2$. The only difference is that $\theta$ is replaced by $\pi/2 \; -\theta$, which accounts
for the difference in sign. \endproof

\begin{lemma}
\la{g:ser1}
Let $(\beta ,\gamma)\in E_2$.  If the point
$x\in\R^2\setminus\{0\}$ satisfies
$\gamma x_1^2 - 2\mu   x_1 x_2 +\gamma x_2^2 =0$,
then its norm is given by
\be
p(x) =\bigg(\frac{ 8 \; q}{ (q+\beta+1)(q-\beta+3)}
\bigg)^{\frac14} |x|.
\la{g:ser2}
\ee
\end{lemma}
\proof Writing $x=|x|(\cos\omega,\sin\omega)$ we have
(cf. proof of Lemma \ref{lem:3.2})
\[
p(x) =\max_{t\in\R} \frac{ \cos(t-\omega)}
{ \Big(  1 +2\gamma \sin(2t) + \frac{\beta-1}{2}\sin^2(2t)
\Big)^{\frac14}}\, |x|.
\]
{\bf Claim.} The above maximum is attained at $t=-\omega$.

To prove this claim it is enough to establish that for all $t\in\R$,
\begin{eqnarray*}
H(t) &:= &\cos^4(2\omega) \Big(  1 +2\gamma \sin(2t) + \frac{\beta-1}{2}\sin^2(2t) \Big) \\
&& \;\; - \cos^4(t-\omega)\Big(  1 -2\gamma \sin(2\omega) + 
\frac{\beta-1}{2}\sin^2(2\omega) \Big) \geq 0 \; .
\end{eqnarray*}
The argument the led to (i) of Lemma \ref{lem:5.3} from assumption
(i) of Lemma \ref{lem:E2} can also be applied to $(x_1,x_2)$ with
only a difference in sign. We thus obtain
\[
\gamma =\frac{\beta+1}{2} \cdot \frac{\sin(2\omega)}{1+\sin^2(2\omega)}.
\]
Morever, applying the elementary identities 
\[
\sin(2\omega) =\frac{ 2\tan \omega}{1+\tan^2 \omega} \; , \qquad 
\cos(2\omega) =\frac{1-\tan^2\omega}{1+\tan^2\omega}
\]
and
\[
\cos^2(t-\omega) = \frac{ (1+\tan t \, \tan\omega)^2}{ (1+\tan^2 t)
(1+\tan^2 \omega)}
\]
we express all trigonometric numbers in $F(t)$ in terms of  $S:=\tan\omega$ and $T:=\tan  t$. We obtain
\[
H(t)= \frac{ (1-S^2)^2(S+T)^2}{ (1+S^2)^4(1+ 6S^2+S^4)(1+T^2)^2} ( A_2 T^2+A_1T+ A_0),
\]
where
\begin{align*}
A_2 \, &= \,  2\, S^6 \beta -11\,  S^4+4S^2 +1 \, ,  \\
A_1 \, & = \, (4\beta+2)\, S^5 -4(\beta+4)\,  S^3  +(4\beta+2)
\, S \, , \\
A_0 \, & = \, S^6 +4\, S^4 -11\, S^2 +2\beta \, .
\end{align*}
The discriminant of the quadratic polynomial is
\[
A_1^2-4A_0A_2= -8(1-S^2)^2 (\beta \, S^4  -6S^2 +\beta)
\Big(S^4+ (4-2\beta)S^2 +1\Big).
\]
Since $\beta>3$ we have $\beta \, S^4  -6S^2 +\beta >0$. So
the Claim will be proved if we show that 
$S^4+ (4-2\beta)S^2 +1 \geq 0$.
For this we note that
\[
\sin(2\omega) =  \frac{ \gamma}{\mu} = \frac{4\gamma}{q+\beta+1}.
\]
Recalling that $q=\sqrt{(\beta +1)^2-16\gamma^2}$ we see that
for fixed $\beta>3$, $\sin(2\omega)$ is an odd and
increasing function of $\gamma$. Since
$\gamma^2<(\beta-1)/2$, we conclude that
\[
\frac{4 \, S^2}{(1+S^2)^2}=
\sin^2(2\omega) \leq \frac{16\, \gamma^2}{ (q+\beta+1)^2}
\bigg|_{\gamma^2=
\frac{\beta-1}{2}}  =\frac{2}{\beta-1}.
\]
Therefore $2(\beta-1)S^2 \leq (1+S^2)^2$ and the Claim follows.

Finally, to compute the maximum value we note that
\[
 1 -2\gamma \sin(2\omega) + \frac{\beta-1}{2}
\sin^2(2\omega) = \frac{ q \, (q-\beta+3)}{2(q+\beta +1)}
\]
and
\be
\cos^2(2\omega)= 1-\sin^2(2\omega)
=\frac{ 2 q}
{ q+\beta+1}.
\la{g:cos2}
\ee
Hence
\[
p^4(x)  =   \frac{\cos^4(2\omega)}{  1 -2\gamma \sin(2\omega) + \frac{\beta-1}{2}\sin^2(2\omega)} 
 =   \frac{ 8 \; q}{ (q+\beta+1)(q-\beta+3)}
\]
and the proof is complete.\endproof

\

\begin{lemma}
Let $(\beta ,\gamma)\in E_2$.
Assume that $x\in\R^2\setminus\{0\}$ satisfies
$\gamma x_1^2 - 2\mu   x_1 x_2 +\gamma x_2^2 =0$. Then there exist
two points $(\pm\xi_0 , \eta_0)\in\R^2\times\R^2$ which are extremal points
for inequality \eqref{g:pol3} and are also critical points of the function
$\phi$ defined by \eqref{PHI}.
\la{g:agg1}
\end{lemma}
\proof
Let us write $x=|x|(\cos\omega ,\sin\omega)$. We then have
\be
\sin(2\omega) = \frac{\gamma}{\mu} \; , \qquad   \cos^2(2\omega)= \frac{2q}{q+\beta+1}.
\la{id9}
\ee
We consider a point $\eta_0=|\eta_0|(\cos\theta ,\sin\theta)$ for which
we assume that
\be
|\cos\theta|=| \cos\omega|\; , \quad |\sin\theta|= |\sin\omega|
\; , \quad \cos\theta  \sin\theta =-\cos\omega \sin\omega \, .
\la{oik}
\ee
Such a point $\eta_0$ satisfies (i) of Lemma \ref{lem:E2}. Hence if $\xi_0$
is taken to satisfy (ii) of the same lemma, then $(\pm\xi_0 , \eta_0)$
are extremal points for inequality \eqref{g:pol3}. We will show that
$\eta_0$ can be chosen so that $(\pm\xi_0 , \eta_0)$ are also
critical points of $\phi$.

Now, writing $z=\xi+ i\eta$ we have $\phi(z)=u(\xi,\eta)
+iv(\xi,\eta)$ where
\[
u(\xi,\eta) = -x\cdot \eta - \frac14
 \, {\rm Re}\, A(\xi +i\eta) \; , \qquad  v(\xi,\eta) = x\cdot \xi 
 - \frac14 \,  {\rm Im}\, A(\xi +i\eta).
\]
From Lemma \ref{lem:E2} follows immediately that
$u_{\xi_j}(\pm\xi_0,\eta_0)=0$, $ j=1,2$. Moreover, by Lemma \ref{lem:5.4} we have
\begin{align*}
&  u_{\eta_1}(\pm\xi_0,\eta_0) 
=- x_1  +\frac{ 2 \, q \, (q+\beta-3)}{ q-\beta +3} \,  |\eta_0|^3 \cos\theta  \cos(2\theta) \\
&  u_{\eta_2}(\pm\xi_0,\eta_0) 
=- x_2 - \frac{ 2 \, q \, (q+\beta-3)}{ q-\beta +3} \,  |\eta_0|^3 \sin\theta  \cos(2\theta) .
\end{align*}
Noting that by \eqref{oik} we have $\cos(2\theta)=\cos(2\omega)$,
we conclude that for $(\pm\xi_0,\eta_0)$ to be critical points we must have
\be
\left\{ 
\begin{array}{l}
 \dfrac{ 2 \, q \, (q+\beta-3)}{ q-\beta +3} \,  |\eta_0|^3 \cos\theta  \cos(2\omega) = |x|\cos\omega 
 \\[0.3cm]
 \dfrac{ 2 \, q \, (q+\beta-3)}{ q-\beta +3} \,  |\eta_0|^3 \sin\theta  \cos(2\omega) =-|x|\sin\omega .
\end{array}
\right.
\la{777}
\ee
Hence we choose $\eta_0$ by
\be
|\eta_0|^3 =  \frac{q-\beta +3}{ 2 \, q \, (q+\beta-3) \; |\cos(2\omega)|} \, |x|
\la{g:agg}
\ee
and
\[
\theta =\darr{ -\omega ,}{ \mbox{ if }\cos(2\omega)>0,}
{\pi -\omega  ,}{ \mbox{ if }\cos(2\omega)<0.}
\]
Then \eqref{777} are valid; this completes the proof.
\endproof

We can now prove the main asymptotic formula.
We define
\[
\sigma_{\beta,\gamma} =
\bigg(\frac{ ( q -\beta +3)^2}{q  \; (q +\beta -3)}
\bigg)^{\frac13} \sigma_2.
\]

\begin{theorem}
Let $(\beta ,\gamma)\in E_2$ and
assume that $x\in\R^2\setminus\{0\}$ satisfies $\gamma x_1^2 - 2\mu   x_1 x_2 +\gamma x_2^2 =0$. Then there holds
\begin{align*}
\hspace{-2.5cm}G(x,t)  = & \;  \frac{2^{\nicefrac13}}{3^{\nicefrac12}\pi} \cdot
\frac{ q^{\nicefrac16}(q+\beta-3)^{\nicefrac16}}{
(q -\beta+3)^{\nicefrac13}(3q+\beta-3)^{\nicefrac12} } \\
& \hspace{1.5cm} \times
\frac{1}{p(x)^{\nicefrac23} \, t^{\nicefrac13}} 
\exp\Big\{\! -\sigma_{\beta,\gamma} \frac{p(x)^{\nicefrac43}}{t^{\nicefrac13}}\Big\}
\big( 1+O(t^{\nicefrac13}) \big)
\end{align*}
as $t\to 0+ \,$.
\la{thm:E2}
\end{theorem}
\proof
Similarly to the proof of Theorem \ref{thm:E1} we may assume that $|x|=1$
and write $x=(\cos\omega,\sin\omega)$. We consider the
function $F(\lambda)$ given by \eqref{fl}
and we are interested in the asymptotic behavior of $F(\lambda)$ as $\lambda\to +\infty$.
So the phase function
is
\be
\phi(z) = i(z_1 \cos\omega  + z_2 \sin\omega ) -\frac14
\big(z_1^4+4 \gamma z_1^3 z_2+ 2\beta z_1^2z_2^2+4 \gamma z_1z_2^3+z_2^4\big).
\la{g:pde}
\ee
We shall prove that the critical points $z_0^{\pm}=\pm\xi_0+i\eta_0$ of
Lemma \ref{g:agg1} are those
that contribute to the asymptotic behavior of $F(\lambda)$.

Arguing as in the proof of Theorem \ref{thm:E1}, we need to establish that
\be
{\rm Re}\, \phi(z) \leq {\rm Re}\,\phi(z_0^\pm), \qquad z\in \mathbb{R}^2+i\eta_0 \, ,
\la{g:agg2}
\ee
with equality exactly at the points $z_0^\pm$. Since $(\pm\xi_0,\eta_0)$
are extremal for Lemma \ref{lem:E2}, we have
\[
{\rm Re}\, A(\xi +i\eta_0) \geq - \kbg A(\eta_0)  \; , \qquad
\xi\in\R^2,
\]
with equality only if $\xi=\pm\xi_0$. It follows that
\[
{\rm Re}\, A(\xi +i\eta_0) \geq {\rm Re}\, A(\xi_0 +i\eta_0)  \; , \qquad
\xi\in\R^2,
\]
with equality if and only if $\xi=\pm\xi_0$. Hence \eqref{g:agg2}
follows.

It follows that the asymptotic behavior of $F(\lambda)$ as
$\lambda\to +\infty$
is given as the sum of two contributions, one from $z_0^+$ and one from $z_0^-$. The two contributions are given by
\[
{\rm contr}(z_0^{\pm}) =  \frac{2\pi}{\lambda} 
\Big( \det(D^2\phi) \Big|_{z=z_0^{\pm}} \Big) ^{-\frac12}
e^{\lambda\phi(z_0^{\pm})}.
\]
By \eqref{g:agg}  and \eqref{id9}
\eqref{g:beli} and  we have
\be
|\eta_0|^4   =   \frac{ (q+\beta +1)^{\nicefrac23}}{4 \, q^2}\Big(  \frac{ q-\beta+3}{q+\beta-3}\Big)^{\nicefrac43}.
\la{g:modeta}
\ee
Hence, recalling also (ii) of Lemma \ref{lem:5.3} we obtain
\begin{align}
A(\eta_0) \;  = \; & \frac{ q\, (q-\beta+3)}{2(q+\beta+1)} 
\times \frac{ (q+\beta+1)^{\nicefrac23}}{4\, q^2} \,
\bigg( \frac{q-\beta+3}{q+\beta-3}\bigg)^{\nicefrac43} \nonumber \\
\;  = \; &  \frac{ (q-\beta+3)^{\nicefrac73}}{8\, q \; 
(q+\beta-3)^{\nicefrac43}(q+\beta+1)^{\nicefrac13}}.
\la{g:arx}
\end{align}

Now, since $z_0^\pm$ is a critical point of $\phi$,
we have $ix =(1/4)\nabla A(z_0^\pm)$. Hence by Euler's theorem 
for homogeneous functions we obtain
\[
\phi(z_0^\pm) = i z_0^\pm\cdot x -\frac14 A(z_0^\pm)
=z_0^\pm \cdot \frac14 \nabla A(z_0^\pm) - \frac14 A(z_0^\pm)
=\frac34 A(z_0^\pm).
\]

We recall that $\eta_0=|\eta_0|(\cos\theta ,\sin\theta)$.
By Lemma \ref{g:agg3} and (i) of Lemma \ref{lem:5.3} we obtain that the imaginary part of $A(z_0^\pm)$ is zero. 
Hence, applying Lemma \ref{lem:E2} we obtain
\[
A(z_0^\pm)  =   {\rm Re} \; A(z_0^\pm)  
 =  -\kbg A(\eta_0).
\]
Combining the above we arrive at
\[
\phi(z_0^\pm) =- \, \frac34 \; \bigg( \frac{ q-\beta +3}{ (q+\beta+1)
(q+\beta-3)}\bigg)^{\frac13}.
\]

Concerning the Hessian of $\phi$ we first compute
\be
\det(D^2\phi) =
3(\beta-3\gamma^2)(z_1^4+z_2^4)-6\gamma (\beta-3) z_1z_2(z_1^2+z_2^2)
+3(6\gamma^2+3-\beta^2)z_1^2z_2^2.
\la{g:hessian}
\ee
Taking in particular $z\in\C^2$ to be of the form
$\pm r(\eta_2,\eta_1) +i(\eta_1,\eta_2)$ (which is the case for
an extremal point for Lemma \ref{lem:E2}) gives
\begin{align*}
\det(D^2\phi) \, = \; & 3\Big[ (\beta-3\gamma^2)(r^4+1)+
(\beta^2-6\gamma^2-3 )r^2\Big] (\eta_1^4+\eta_2^4) \\
& -6\gamma(\beta-3)(r^4-6r^2+1)\eta_1\eta_2(\eta_1^2+ \eta_2^2) \\
& +3\Big[ (6\gamma^2+3-\beta^2)(r^4+1) +4(\beta^2-3\beta-3 
+3\gamma^2)r^2\Big] \eta_1^2\eta_2^2 \\
& -6 \,i \, (\beta-3)(r^3-r)\bigg\{ \gamma (\eta_1^4 +\eta_2^4) 
+(\beta+1)\eta_1\eta_2(\eta_1^2+\eta_2^2)
 +6\gamma \eta_1^2\eta_2^2\bigg\}.
\end{align*}
We now take $\eta=\eta_0$.
The imaginary part of $\det(D^2\phi)$ is equal to
$\gamma  +\frac{\beta+1}{2}\sin(2\theta)+\gamma^2\sin^2(2\theta)$
and therefore vanishes by (i) of Lemma~\ref{lem:5.3}.
After some computations we find
\begin{align*}
\det(D^2\phi) \, = \; & |\eta_0|^4 \, \bigg\{
3\Big[ \, (\beta-3\gamma^2)(r^4+1)+
(\beta^2-6\gamma^2-3)r^2\Big]  \\
& -3\gamma(\beta-3)(r^4-6r^2+1) \sin(2\theta)
+\Big[ \Big( \frac94 +9\gamma^2 - \frac34 \beta^2-\frac32 \beta
\Big)(r^4+1) \\
&  +\Big( \frac32 \beta^2 -9\beta-\frac92
+18\gamma^2\Big)r^2   \Big]\sin^2(2\theta) \bigg\}.
\end{align*}
Substituting for $r^2$ and $\sin(2\theta)$ using \eqref{g:bo1} and 
\eqref{g:sif} respectively we arrive after some simplifications at
\[
\det(D^2\phi)\Big|_{z=z_0^\pm} = \frac32 (3q+\beta-3)
\bigg( \frac{q-\beta+3}{(q+\beta-3)(q+\beta+1)}\bigg)^{\frac13}.
\]

Hence the two contributions are equal and therefore
\begin{align}
  F(\lambda) \; \sim \; &
 {\rm contr}(z_0^+)+ {\rm contr}(z_0^-) \nonumber \\
\; = \; &  \frac{4\pi}{\lambda}\left[ \det(\phi_{z_iz_j})\mid_{z=z_0^+}\right]^{-\frac{1}{2}} \exp\lbrace \lambda \phi(z_0^+)
\rbrace  
\nonumber\\
\; = \; & \frac{4\sqrt{2} \, \pi}{\sqrt{3}} \;
\frac{ (q+\beta-3)^{\nicefrac16}(q +\beta+1)^{\nicefrac16}}
{(q -\beta+3)^{\nicefrac16}(3q+\beta-3)^{\nicefrac12} } \nonumber \\
 & \;\;  \times \; \frac{1}{\lambda} \, 
\exp\Big\{
- \, \frac34 \; \Big( \frac{ q-\beta +3}{ (q+\beta+1)
(q+\beta-3)}\Big)^{\nicefrac13} \lambda
\Big\}.
\la{clim}
\end{align}

Applying Lemma \ref{g:ser1} we find that
\[
\lambda=\frac{1}{(4t)^{\nicefrac13}} =\frac12 \, \bigg(
\frac{ (q+\beta+1)(q-\beta+3)}{4q}\bigg)^{\nicefrac13}
\frac{p(x)^{\nicefrac43}}{t^{\nicefrac13}}
\]
and therefore
\[
\exp\Big\{
- \, \frac34 \; \bigg( \frac{ q-\beta +3}{ (q+\beta+1)
(q+\beta-3)}\bigg)^{\nicefrac13} \lambda
\Big\} =
\exp\Big\{\! -\sigma_{\beta,\gamma} \frac{p(x)^{\nicefrac43}}{t^{\nicefrac13}}\Big\}.
\]
We  similarly find that
\[
\frac{1}{\lambda} =2^{\nicefrac76} \bigg( \frac{q}{(q+\beta+1)
(q-\beta+3)}\bigg)^{\nicefrac16}\frac{t^{\nicefrac13}}{p(x)^{
\nicefrac23}}.
\]
We  thus conclude that
\begin{align*}
& G(x,t) \; = \;     \frac{1}{(2\pi)^2 (4t)^{\nicefrac{2}{3}}}
F(\lambda )  \\
\; = \; & \frac{2^{\nicefrac13}}{3^{\nicefrac12}\pi} 
\frac{ q^{\nicefrac16}(q+\beta-3)^{\nicefrac16}}{
(q -\beta+3)^{\nicefrac13}(3q+\beta-3)^{\nicefrac12} } 
\frac{1}{p(x)^{\nicefrac23}\, t^{\nicefrac13}} 
\exp\Big\{\! -\sigma_{\beta,\gamma} \frac{p(x)^{\frac43}}{t^{\frac13}}\Big\}
\big( 1+O(t^{\nicefrac13}) \big)
\end{align*}
as required.
\endproof

\section{The case $(\beta,\gamma) \in L_2$}

We recall that $(\beta,\gamma) \in L_2$ if and only if
\[
\beta = 3 \quad  \mbox{ and } \quad   \gamma^2 <\frac{\beta-1}{2}.
\]
Similarly to the case 
$(\beta,\gamma)\in L_1$, there are now four relevant saddle points for 
each of the  points $x\in\R^2$ that we consider.

We define the constants
\[
q=4\sqrt{1-\gamma^2}  \; , \qquad \mu =1+ \sqrt{1-\gamma^2}.
\]
These are the same constants as those of
 Section \ref{sec:E2}, cf.
\eqref{q}, extended by continuity for $\beta=3$.

\begin{lemma}
\la{lem:L2}
Assume that $(\beta ,\gamma)\in L_2$. There holds
\be
\la{g:pol31}
{\rm Re}\, A(\xi +i\eta) \geq - 8 A(\eta)  \; , \qquad
\xi, \; \eta\in\R^2.
\ee
Moreover the constant 8 is the best possible and equality 
holds  precisely for the points $(\xi,\eta)\in $ such that
\begin{align*}
{\rm (a)} & \hspace{1cm} \ia    \mbox{ the point $\eta\in\R^2$ satisfies $
\gamma\eta_1^2 +2\mu\eta_1\eta_2 +\gamma\eta_2^2 =0$} 
\nonumber \\
& \;\, \mbox{ and }\ib  \mbox{ the point $\xi$ is one of 
the two points $\pm (\sqrt{3}\eta_2 ,\sqrt{3}\eta_1 )$} \\
\mbox{ or  } \hspace{0.6cm} 
 {\rm (b)} & \qquad\quad  \xi =\pm \sqrt{3} \, \eta.
\end{align*}
\end{lemma}
\proof The constant $\kbg$ of Lemma \ref{lem:E2} is equal
to 8 when $\beta=3$, hence \eqref{g:pol31} follows by continuity.

It remains to establish when exactly we have equality.  Using again Lemma \ref{lem:E2} we obtain that if the point $(\xi,\eta)$ satisfies
(a) then we do have equality. The fact that for $\xi=\pm\sqrt{3}\eta$
we have equality follows from the computations of \cite{EP1}
or by simplify verifying that the vector $[ 3\eta_1^2 \, 3\eta_1\eta_2
,3\eta_2^2 ,1]^{T}$ belongs in the kernel of the matrix $Q(\eta)$, cf.
proof of Lemma 1.2.

\begin{theorem}
Supppose that $(\beta,\gamma)\in L_2$. Assume that $x\in\R^2\setminus\{0\}$ satisfies $\gamma x_1^2 - 2\mu   x_1 x_2 +\gamma x_2^2 =0$. Then there holds
\begin{align*}
G(x,t) \; \sim \; & \;  \frac{2^{\nicefrac13}}{3 \, \pi \, q^{\nicefrac12}} 
\cdot \frac{1}{p(x)^{\nicefrac23}\, t^{\nicefrac13}} \,
\exp\Big\{ - \sigma_2 \frac{ p(x)^{\nicefrac43}}{t^{\nicefrac13}}\Big\} \\
& \quad \times  \left(  1+  \Big(  \frac{q+4}{2q}\Big)^{\nicefrac12}
\cos\bigg(  \sqrt{3}\,\sigma_2 \, \dfrac{p(x)^{\nicefrac{4}{3}}}{t^{\nicefrac{1}{3}}} -\frac{\pi}{3}\bigg) +O(t^{\nicefrac13})  \right)
\end{align*}
as $t\to 0+$.
\end{theorem}
\proof We argue as in the proof of Theorem \ref{thm:E2}; in particular
we write $x=(\cos\omega ,\sin\omega)$ and
consider the functions $F(\lambda)$ and $\phi(z)$ defined by
\eqref{fl} and \eqref{g:pde} respectively.

Setting $\beta=3$, the points $z_0^{\pm}$ used in the proof of
Theorem \ref{thm:E2} are critical points of $\phi(z)$ by continuity.
The points
\[
z_*^+=  (\sqrt{3} +i)\eta_0 = 2e^{\frac{\pi i}{6}} \eta_0 \;
, \qquad  
z_*^-=  (-\sqrt{3} +i)\eta_0 = 2e^{\frac{5\pi i}{6}} \eta_0
\]
are also critical points.

Applying Lemma \ref{lem:L2} we see that these four points
are precisely the points where ${\rm Re \, \phi(z)}$ attains its
maximum when restricted  in $\R^2+i\eta_0$; hence it is these points
that contribute to the asymptotic behavior of $F(\lambda)$.

The contribution of the points $z_0^{\pm}$ is obtained by letting
$\beta\to 3+$ at the corresponding contribution of
Theorem \ref{thm:E2} (cf. \eqref{clim}), so 
\[
{\rm contr}(z_0^+)+ {\rm contr}(z_0^-) = \frac{4\sqrt{2} \; \pi}{3}
\cdot \frac{ (q+4)^{\nicefrac16}}{q^{\nicefrac12}} \cdot \frac{1}{\lambda}
\cdot
\exp\Big\{ -\frac34 \, \frac{1}{ (q+4)^{\nicefrac13}} \lambda\Big\}.
\]
Concerning the point $z_*^+ =2e^{ \frac{\pi i}{6}} \eta_0$
we first note that, by  \eqref{g:arx},
\[
A(\eta_0) =\frac{1}{8 \, (q+4)^{\frac13}}.
\]
Hence
\[
\phi(z_*^+)  =  \frac34 A(z_*^+) 
=\frac34 \cdot 16 \;  e^{\frac{2\pi i}{3}} A(\eta_0) 
=\frac32  \, \frac{e^{\frac{2\pi i}{3}}}{(q+4)^{\frac13}} \, .
\]
Now, to compute the Hessian at $z_*^+$ we first note that for
$\beta=3$ relation \eqref{g:hessian} becomes
\[
\det(D^2\phi) = 9(1-\gamma^2) (z_1^4 -2z_1^2z_2^2 +z_2^4)
= \frac{9}{16} \, q^2 \, (z_1^4 -2z_1^2z_2^2 +z_2^4). 
\]
Hence at the point $z_*^+ =2e^{ \frac{\pi i}{6}} \eta_0$,
\begin{align*}
\det(D^2\phi)\Big|_{z=z_0^+}
\; = \;   & 9 \, q^2
e^{ \frac{2\pi i}{3}} (\eta_{01}^4 -2\eta_{01}^2\eta_{02}^2 
+\eta_{02}^4) \\
\; = \;   & 9 \, q^2
e^{ \frac{2\pi i}{3}} |\eta_0|^4 \Big(
\cos^4\theta  -2\cos^2\theta \sin^2\theta+\sin^4\theta\Big)
\end{align*}
Setting $\beta=3$ in \eqref{g:modeta} we get
$|\eta_0|^4 = (q+4)^{\frac23}/4q^2$.
Moreover, by \eqref{g:cos2},
\[
\cos^4\!\theta  -2\cos^2\!\theta \;  \sin^2\!\theta+\sin^4\!\theta =\cos^2(2\theta)
= \cos^2(2\omega)=\frac{ 4 q^2}
{ (q+4)^2}.
\]
Hence
$\det(D^2\phi)|_{z=z_0^+} =9 \, q^2 (q+4)^{-\frac43}
e^{ \frac{2\pi i}{3}} $ and we conclude that
\begin{align*}
{\rm contr}(z_*^+) \; = \; & 2\pi 
\Big( 9 \, q^2 (q+4)^{-\frac43} 
e^{ \frac{2\pi i}{3}}\Big)^{-\frac12} \cdot
\frac{1}{\lambda} \cdot 
 \exp\Big\{\frac{3}{2} \frac{ e^{\frac{2\pi i}{3}}}{(q+4)^{\frac13}} \, \lambda\Big\}  \\
\; = \; & \frac{2\pi}{3}  \cdot 
 \frac{(q+4)^\frac23}{q}
  \cdot \frac{1}{\lambda} \cdot 
 \exp\Big\{ -\frac{3}{4\,  (q+4)^{\frac13}}\lambda\Big\}
 \exp\Big\{\Big( \frac{3\sqrt{3}}{4 \, (q+4)^{\frac13}}
 \lambda -\frac{\pi}{3}
 \Big) i\Big\}.
 \end{align*}
The contribution of $z_*^-$ is the complex conjugate of that
of $z_*^+$, hence
\begin{align*}
{\rm contr}(z_*^+)+ {\rm contr}(z_*^-)\;  = \; & 
\frac{4\pi}{3} \cdot  \frac{(q+4)^\frac23}{q}
\cdot \frac{1}{\lambda} \cdot
\exp\Big\{ -\frac{3}{4\,  (q+4)^{\frac13}}\lambda\Big\} \\
& \quad \times \cos\Big(  \frac{3\sqrt{3}}{4 \, (q+4)^{\frac13}}
 \lambda -\frac{\pi}{3}\Big).
\end{align*}
Adding up all contributions we arrive at
\begin{align*}
F(\lambda) \; = \; & \;  \frac{4\sqrt{2} \; \pi}{3}
\cdot \frac{ (q+4)^{\nicefrac16}}{q^{\nicefrac12}} \cdot \frac{1}{\lambda}
\cdot
\exp\Big\{ -\frac34 \, \frac{1}{ (q+4)^{\nicefrac13}} \lambda\Big\} \\
& \quad \times  \bigg(  1+  \Big(  \frac{q+4}{2q}\Big)^{\nicefrac12}
\cos\Big(  \frac{3\sqrt{3}}{4 \, (q+4)^{\frac13}}
 \lambda -\frac{\pi}{3}\Big) +O(\frac{1}{\lambda}) \bigg).
\end{align*}
By \eqref{g:ser2} we have
\[
p(x) =\Big(\frac{ 8 }{q+4}\Big)^{\frac14} ,
\]
hence substituting $\lambda=1/(4t)^{1/3}$ we find that
\[
\frac{\lambda}{(q+4)^{\nicefrac13}} = \frac{1}{2 \cdot 4^{\nicefrac13}} \, 
\frac{p(x)^{\nicefrac43}}{t^{\nicefrac13}}.
\]
This we substitute in the exponential and in the cosine. 
Similarly we have
\[
\frac{1}{\lambda} =\frac{2^{\nicefrac76} }{(q+4)^{\nicefrac16}} 
\frac{t^{\nicefrac13}}{p(x)^{\nicefrac23}}  ,
\]
which we substitute for the factor preceding the exponential.
Recalling \eqref{g:GF} completes the proof.
\endproof

\

\section{Heat kernel estimates}

In this section we apply Lemmas \ref{lem:E1}
and \ref{lem:E2} in order to obtain Gaussian estimates with best
constant on the heat kernel of variable coefficient operators whose 
principal symbol is given by \eqref{pr:sy}.

Let $\Omega \subseteq \mathbb{R}^2$ be open and connected and consider a differential operator $H$ on $L^2(\Omega)$ (complex-valued functions) given formally by
$$
\begin{aligned}
Hu(x) = & \; \partial^4_{x_1}u + 2\partial^2_{x_1}\big( \gamma(x) \partial^2_{x_1x_2} u\big) + 2\partial^2_{x_1x_2}\left( \gamma(x) \partial^2_{x_1} u\right) 
+ 2\partial^2_{x_1x_2}\left(\beta(x)\partial^2_{x_1x_2}u\right) \\
&+ 2\partial^2_{x_2^2}\left( \gamma(x) \partial^2_{x_1x_2} u\right) + 2\partial^2_{x_1x_2}\left( \gamma(x) \partial^2_{x_2^2} u\right) + \partial^4_{x_2}u
\end{aligned}
$$
Here $\beta,\gamma\in L^{\infty}(\Omega)$ and in case $\Omega \neq \mathbb{R}^2$ we impose Dirichlet boundary conditions
on $\partial \Omega$.
The operator $H$ is defined by means of the sesquilinear form
\begin{align*}
Q(u,v) =& \int\limits_\Omega
\big\lbrace u_{x_1x_1}\overline{v_{x_1x_1}}
+  2\gamma(x) \big( u_{x_1x_2}\overline{v_{x_1x_1}} + u_{x_1x_1}\overline{v_{x_1x_2}} \big) \\
& + \, 2\beta(x)u_{x_1x_2}\overline{v_{x_1x_2}} +
2\gamma(x) \big( u_{x_1x_2}\overline{v_{x_2x_2}} +
 u_{x_2x_2}\overline{v_{x_1x_2}} \big) +
u_{x_2x_2}\overline{v_{x_2x_2}} \big\rbrace \, dx
\end{align*}
defined on $H^2_0(\Omega)$. 
We assume that $H$ is uniformly elliptic in the sense
that G\r{a}rding's inequality
\[
Q(u) \geq c_1 \| \Delta u\|_{L^2(\Omega)}^2 -c_2
 \|u\|_{L^2(\Omega)}^2 \; , \qquad u\in H^2_0(\Omega),
\]
is valid. This implies, cf. \cite[Theorem 7.12]{A},
that the associated symbol
\[
 A(x,\xi) =\xi_1^4 + 4\gamma(x) \,
 \xi_1^3 \, \xi_2 +  2\beta(x)\, \xi_1^2 \, \xi_2^2 + 
 +4\gamma(x) \, \xi_1^3 \, \xi_2 +\xi_2^4 \;  , \quad x\in\Omega , \; 
 \xi\in\R^2,
\]
satisfies
\[
A(x,\xi) \geq c  \, |\xi|^4
\]
for some $c>0$ and all $x\in\Omega$ and $\xi\in\R^2$.
The form $Q$ is then closed and Hermitian and we define
$H$ as the associated self-adjoint operator. As already mentioned, the 
operator $H$ has a continuous heat kernel $G(x,x',t)$ which satisfies the 
Gaussian estimate \eqref{gauss}.

The main theorem of this section contains a sharp version of 
\eqref{gauss}. In order to state it we need some definitions. Firstly,
we define the class of real-valued functions
$$
\mathcal{E}=\left\lbrace \varphi \in C^2(\Omega): 
\| D^{\alpha}\varphi \|_\infty <+\infty, \;\; 0 \leq |\alpha | 
\leq 2 \right\rbrace 
$$
(here standard multiindex notation is used for
$D^{\alpha}\varphi$); we also define the subclass
$$
\mathcal{E}_{A,M}=\{ \varphi \in \mathcal{E}: A(y,\nabla \varphi (y))\leq 1 \ , \; y \in \Omega \mbox{ and } 
\| D^{\alpha}\varphi \|_\infty \leq M \ , \ |\alpha | = 2 \}. 
$$
We next define for $M>0$  a distance $d_M( \cdot , \cdot )$ on $\Omega$ by
$$
d_M( x , x' ) = \sup \left\lbrace \varphi(x') - \varphi(x) : \varphi 
\in \mathcal{E}_{A,M} \right\rbrace.
$$
Finally, we denote by $D$ the distance in $L^\infty (\Omega)$ of the functions $\beta,\, \gamma$ from the space ${\rm Lip(\Omega)}$ of all Lipschitz continuous functions,
\[
D=\max\big\{d_{L^\infty} (\beta, {\rm Lip(\Omega)}) \, , \; d_{L^\infty} (\gamma, {\rm Lip(\Omega)})\big\}.
\]
Given the above, we can now state the following

\begin{lemma}
Assume that $k_0>0$ is such that
\[
{\rm Re \,}A(x,\xi +i\eta) \geq -k_0 A(x,\eta) \; , \qquad x\in\Omega\, , \;\; \xi,  \eta\in\R^2,
\]
and let $\sigma_0 =3\cdot 4^{-\nicefrac43} k_0^{-\nicefrac13}$. Then 
the following Gaussian estimate is satisfied: given $\epsilon$
and $M>0$ there exist $c_\epsilon,c_{\epsilon,M}<\infty$ such that
\[
|G(x,x',t)| \leq c_\epsilon t^{-\nicefrac{1}{2}}\exp\Big( -(\sigma_0-cD-\epsilon) \dfrac{d_M(x,x')^{\nicefrac{4}{3}}}{t^{\nicefrac{1}{3}}} +c_{\epsilon,M} t\Big)
\]
for all $x,x' \in \Omega$ and $t>0$.
\la{GB}
\end{lemma}
For the proof we refer to \cite{B2001} or \cite{BB} where the 
proof is divided into a number of steps; the precise form
of the operator or the value of $k_0$ plays no role (but the form
of the operator determines $k_0$).

Thinking of $D$ above as being small, our main theorem below involves an 
exponential constant which is a small pertubation of a sharp
constant $\sigma_*$. The constant $\sigma_*$ reduces to the constant
$\sigma_{\beta,\gamma}$ of Sections \ref{sec:E1} and \ref{sec:E2}
in the case of operators with constant coefficients. Hence to show
the sharpness it is enough to establish how the distances $d_M(x,x')$
are related to the Finsler norms of Sections \ref{sec:E1} and \ref{sec:E2}.
It is easy to see that, as $M \rightarrow +\infty $, the distance
$d_M(x,x')$ converges to the Finsler-type distance
$$
d( x , x' ) = \sup \{ \varphi(x') - \varphi(x) : \varphi \in {\rm Lip(\Omega)}, \;  A(y,\nabla \varphi (y))\leq 1 \ , \; y \in \Omega \}.
$$
We then have
\begin{lemma}
Assume that $H$ is a constant coefficient operator acting on
$L^2(\R^2)$. Then for any $x , x'\in\R^2$,
\[
d(x,x')= p(x-x') \; .
\]
\end{lemma}
\proof Assuming with out loss of generality that $x'=0$, we must
prove that
\[
 \sup \{ \varphi(x) : \varphi \in {\rm Lip(\Omega)}, \;
\varphi(0)=0 , \;  A(\nabla \varphi (y))\leq 1 \ , {\rm a.e.}\; y \in \Omega \} = \max_{\xi \neq 0} \frac{x\cdot \xi}{ A(\xi)^{\nicefrac14}}.
\]
Let $\xi\in\R^2$. The function $\varphi(y) =(y\cdot \xi)
 /A(\xi)^{\nicefrac14}$ belongs to the set above, and this proves
that $d(x,0)\geq p(x)$.
For the converse, let $\varphi$ be a function in the above set. Then
\begin{align*}
\varphi(x) =& \; \int_0^1 \frac{d}{ds} \varphi(sx)\, ds 
= \int_0^1 \nabla\varphi(sx)\cdot x \, ds 
\leq   \; \int_0^1  A( \nabla\varphi(sx))^{\nicefrac14} p(x) \, ds \\
\leq  & \;  \int_0^1  p(x) \, ds = p(x);
\end{align*}
hence $d(x,0)\leq p(x)$, as required. 
\endproof

We now join together the definitions of $k_{\beta,\gamma}$
from the previous sections: given $(\beta,\gamma)\in E$ we define
\[
k_{\beta,\gamma} =\left\{
\begin{array}{ll}  
8, & \mbox{ if } (\beta,\gamma) \in E_0 \cup L_1 \cup L_2 \, ,\\[0.2cm]
 \dfrac{ 8\big( 1-\beta\big)+16\gamma^2}{\big(
\beta + 1 \big)^2-16\gamma^2}, & \mbox{ if } (\beta,\gamma) \in E_1 \, , \\[0.3cm]
\dfrac{8\, q\big(q+\beta-3\big)}{\big(q-\beta +3\big)^2}
, & \mbox{ if } (\beta,\gamma) \in E_2.
\end{array}
\right.
\]
where, we recall,  $q=\big( (\beta+1)^2-16\gamma^2 \big)^{\nicefrac12}$.

We then define on $\Omega$ the function $k(x) = k_{\beta(x),\gamma(x)}$.
We repeat the same procedure for the constants $\sigma_{\beta,\gamma}$
and arrive in the same way at a function $\sigma(x)$, $x\in\Omega$.
Finally, we set
\[
k_*=\sup\limits_{x\in \Omega}k(x)  \; \; , \qquad
\sigma_*=\inf\limits_{x\in \Omega}\sigma(x).
\]
The following is the main theorem of this section.
\begin{theorem}
For all $\epsilon \in (0,1)$ and all $M>0$ large there exist $c_\epsilon,c_{\epsilon,M}<\infty$ such that
\[
|G(x,x',t)| \leq c_\epsilon t^{-\nicefrac{1}{2}}\exp\left( -(\sigma_*-cD-\epsilon) \dfrac{d_M(x,x')^{\nicefrac{4}{3}}}{t^{\nicefrac{1}{3}}} +c_{\epsilon,M} t\right)
\]
for all $x,x' \in \Omega$ and $t>0$.
\begin{proof}
Let $x\in\Omega$. Depending on which one
of the sets $E_0,E_1,E_2,L_1,L_2$ the point $(\beta(x),\gamma(x))$
belongs to, we use one of Lemmas \ref{lem:E1}, \ref{lem:L1}, 
\ref{lem:E2} or \ref{lem:L1} and obtain that for
any $\xi,\eta\in\R^2$ we have
\[
{\rm Re \,}A(x,\xi +i\eta) \geq -k(x) A(x,\eta) 
\geq -k_* A(x,\eta).
\]
We now apply Lemma \ref{GB}. The result will follow if we establish
that $\sigma_* = 3\cdot 4^{-\nicefrac43} k_*^{-\nicefrac13}$. And for
this it is enough to show that that for any $(\beta,\gamma)\in E$
we have
\[
\sigma_{\beta,\gamma} = 3\cdot 4^{-\nicefrac43}
k_{\beta,\gamma}^{-\nicefrac13}.
\]
The last relation is easily seen to be true in all cases, thus completing
the proof.
\end{proof}
\end{theorem}

\newpage

The asymptotic formulas of the four main theorems follows from a
relation of the form $F(\lambda)\sim G(\lambda)$ for some explicitly given function $G(\lambda)$.
In each of the diagrams below we have plotted the numerically computed graph of $F(\lambda)e^{\kappa \lambda}$
(blue, dashed) against the function $G(\lambda)e^{\kappa \lambda}$ (red, continuous),
where $\kappa$ is the positive constant in the exponential factor of
$G(\lambda)$. 

\begin{figure*}[ht]
\centering
\begin{minipage}{.5\textwidth}
  \centering
  \includegraphics[width=1\linewidth]{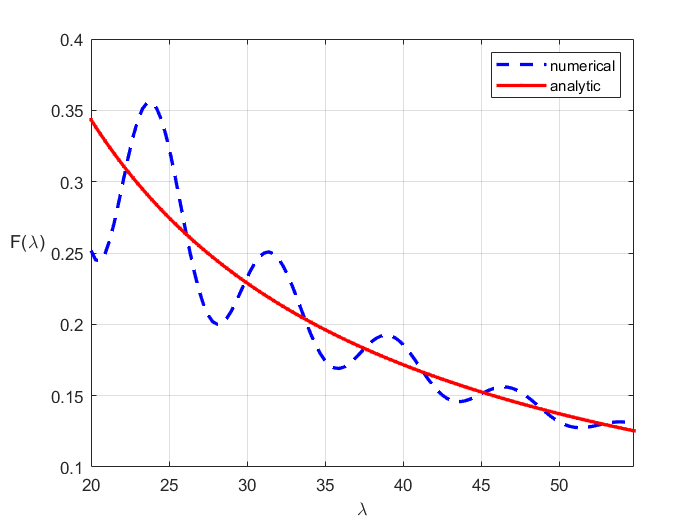}
  \caption{\newline $\beta=1$, $\gamma=12/25$}
\end{minipage}%
\begin{minipage}{.5\textwidth}
  \centering
  \includegraphics[width=1\linewidth]{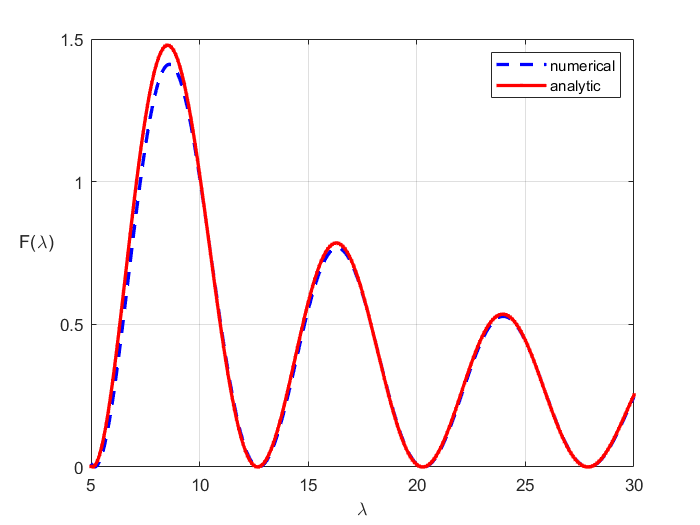}
  \caption{\newline $\beta=0$, $\gamma=\sqrt{2}/3$}
\end{minipage}
\end{figure*}

\begin{figure*}[ht]
\centering
\begin{minipage}{.5\textwidth}
  \centering
  \includegraphics[width=1\linewidth]{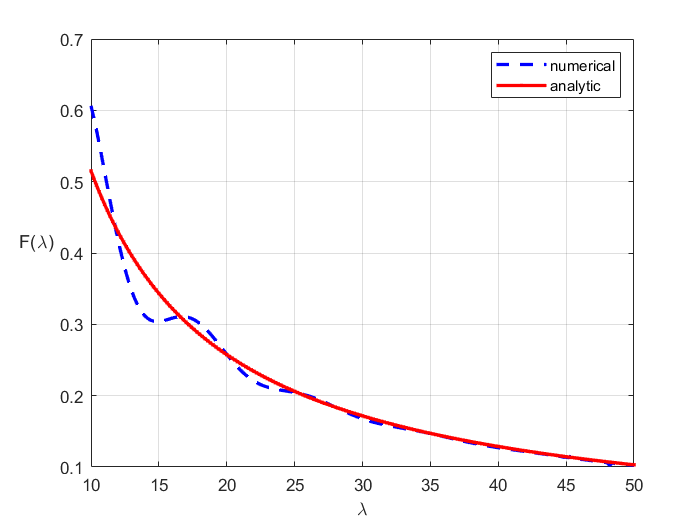}
  \caption{\newline $\beta=3$, $\gamma=1$}
\end{minipage}%
\begin{minipage}{.5\textwidth}
  \centering
  \includegraphics[width=1\linewidth]{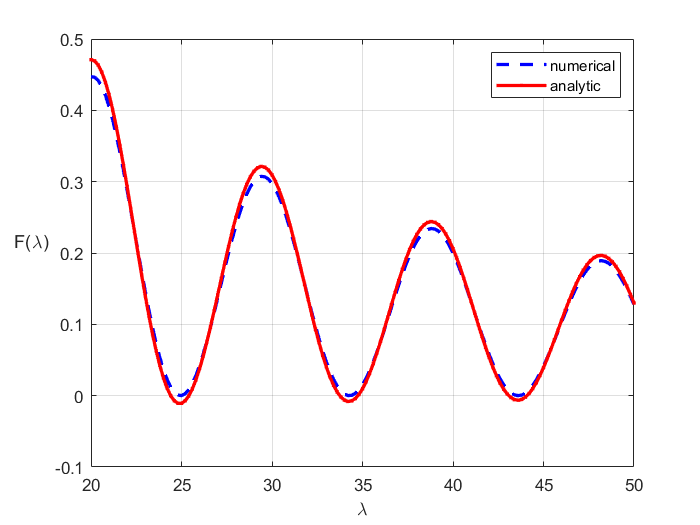}
  \caption{\newline$\beta=3$, $\gamma=3/5$}
\end{minipage}
\end{figure*}

\appendix
\section{Proof of $Q(\eta)X(\xi)=0$}

The notation in the following lemma is that of 
Lemma \ref{lem:E2}.

\begin{lemma}
\la{lem:app}
Assume that $\eta\in\R^2$ satisfies
$\gamma\eta_1^2 +2\mu\eta_1\eta_2 +\gamma\eta_2^2 =0$ and
that $r>0$ is defined by \eqref{g:bo1}. Let
$(\xi_1,\xi_2) =r(\eta_2,\eta_1)$. Then $Q(\eta)X(\xi)=0$.
\end{lemma}
\proof
Without loss of generality we may assume that $|\eta|=1$.
Writing $\eta=(\cos\theta , \sin\theta)$ it then follows that
\be
\sin(2\theta) =-\frac{\gamma}{\mu}  = 
-\frac{4\gamma}{q+ \beta+1}.
\la{ll}
\ee
We express the entries of $Q(\eta)=\{ q_{ij}(\eta)\}$ which
depend on $\eta$ in terms of $2\theta$: since
\[
 \eta_1^2 = \frac{1+\cos(2\theta)}{2} \; , \qquad
 \eta_1\eta_2=\frac{\sin(2\theta)}{2} \; , \qquad
  \eta_2^2 = \frac{1-\cos(2\theta)}{2} \, ,
\]
we have
\[
\begin{array}{l}
q_{14} =  \frac{\beta-3}{2}\cos(2\theta) -3\gamma \sin(2\theta)  - \frac{\beta+3}{2} \\[0.2cm]
q_{24} = -6\gamma -2\beta\sin(2\theta) \\[0.2cm]
q_{34} =-\frac{\beta-3}{2}\cos(2\theta) -3\gamma \sin(2\theta)  - \frac{\beta+3}{2} \\[0.2cm]
q_{44} =(\kbg +1)\Big( 1 + \dfrac{\beta-1}{2}\sin^2(2\theta)
+2\gamma\sin(2\theta)\Big)
\end{array}
\]
We also have
\begin{eqnarray*}
X(\xi) &=&[ r^2\eta_2^2 \; , \;   r^2 \eta_1\eta_2 \; , \;   r^2\eta_1^2  \; , \; 1 ]^{T} \\[0.2cm]
&=& \Big[\frac{r^2\big(1-\cos(2\theta)\big)}{2} 
\; , \; \frac{r^2\sin(2\theta)}{2}
\; , \;   \frac{r^2\big(1+\cos(2\theta)\big)}{2} \; , \; 1 \Big]^T.
\end{eqnarray*}
Let $v=Q(\eta)X(\xi)$. We have
\begin{eqnarray*}
v_1 &=& \frac{r^2}{2}\big(1-\cos(2\theta)\big) +\gamma r^2\sin(2\theta)
 + \frac{r^2\big(1+\cos(2\theta)\big)\big(\beta^2-5\beta+6-(\beta-6)q\big)}{2(\beta-3+3q)} \\
 && +  \frac{\beta-3}{2}\cos(2\theta) -3\gamma \sin(2\theta)  - \frac{\beta+3}{2} \\
 &=&\cos(2\theta) \Big[-\frac{r^2}{2} + 
 \frac{r^2\big(\beta^2-5\beta+6-(\beta-6)q\big)}{2(\beta-3+3q)}  +\frac{\beta-3}{2}    \Big]  \\
&& + \sin(2\theta) (r^2 \gamma 
-3\gamma) 
 + \frac{r^2}{2}+\frac{r^2\big(\beta^2-5\beta+6-(\beta-6)q\big)}{2(\beta-3+3q)} - \frac{\beta+3}{2}.
\end{eqnarray*}
Substituting $r^2$ from \eqref{g:bo1} we find that
the square bracket that multiplies $\cos(2\theta)$ is equal to zero.
Moreover, applying \eqref{g:bo1} and \eqref{g:sif}
we see that the remaining terms also add to zero.
Concerning $v_2$ we have
\begin{eqnarray*}
v_2&=& 2\gamma  \frac{r^2\big(1-\cos(2\theta)\big)}{2} 
+\frac{4\beta-12 +(8\beta-12)q}{3q+\beta-3} \cdot 
\frac{r^2\sin(2\theta)}{2} \\
&& + \, 2\gamma \frac{r^2\big(1+\cos(2\theta)\big)}{2} -6\gamma
-2\beta\sin(2\theta) \\
&=&2\gamma(r^2-3) +\Big(
\frac{r^2\big(4\beta-12 +(8\beta-12)q\big)}{2(3q+\beta-3)} -
2\beta\Big)\sin(2\theta).
\end{eqnarray*}
Substituting for $r^2$ and $\sin(2\theta)$ from \eqref{g:bo1} and \eqref{ll}
we find that this is zero.

It is readily seen that $v_3=v_1$ and therefore $v_3=0$.

Finally, substituting also $\cos^2(2\theta)=1-\sin^2(2\theta)$, we have
\begin{eqnarray*}
v_4&=& \Big(   \frac{\beta-3}{2}\cos(2\theta) -3\gamma\sin(2\theta)
-\frac{\beta+3}{2}\Big)  \frac{r^2\big(1-\cos(2\theta)\big)}{2} \\
&&+\big( -6\gamma-2\beta\sin(2\theta)\big)\, \frac{r^2\sin(2\theta)}{2} \\
&& + \Big(   -\frac{\beta-3}{2}\cos(2\theta) -3\gamma\sin(2\theta)
-\frac{\beta+3}{2}\Big)  \frac{r^2\big(1+\cos(2\theta)\big)}{2} \\
&& +(\kbg +1)\Big( 1+\frac{\beta-1}{2}\sin^2(2\theta)
+2\gamma\sin(2\theta)\Big) \\
&=&  \Big(  \frac{ (\kbg+1)(\beta-1)}{2} -\frac{\beta+3}{2}r^2
\Big)\sin^2(2\theta) \\
&& + \big( -6\gamma r^2+2\gamma(\kbg+1) \big)\sin(2\theta)
+\kbg+1-\beta r^2
\end{eqnarray*}
Once again, this vanishes after we substitute for $r^2$
and $\sin(2\theta)$ from \eqref{g:bo1} and \eqref{ll}. This comletes the proof.
\endproof

\section{The strongly convex case}

We present here a sketch of the proof of Efgrafov et al. \cite{EP1}
in the strongly convex case $(\beta,\gamma)\in E_0$. 

Let $x\in\R^2\setminus\{0\}$. By the strict convexity of the symbol there exists a unique 
$q=q(x)\in\mathbb{R}^2$ such that $\frac{1}{4}\nabla A(q) =x$. The points
\[
z_*^{\pm} = e^{\pm\frac{\pi i}{3}} q   =: \pm \xi_0 +i\eta_0 \, ,
\]
are then critical points of $\phi$ and, we will see, are those that contribute
to the asymptotic behavior of $F(\lambda)$. To see this, we once more
change the domain of integration from $\mathbb{R}^2$ to $\mathbb{R}^2+i\eta_0$ and the 
crucial property is that
\[
{\rm Re} \, A(\xi +i\eta_0 ) + 8A(\eta_0)\geq 0  , \qquad \xi\in\mathbb{R}^2,
\]
with equality only at $\xi=\pm\xi_0$, cf. Lemma\ref{lem:sc}.
We also have
\[
{\rm det} (\phi_{z_iz_j})|_{z=z_*^{\pm}} = h(x)^{\nicefrac{4}{3}} 
e^{\pm\nicefrac{2\pi i}{3}}
\]
where the function $h$ is positively homogeneous of degree one. Hence
\[
{\rm contr}(z_*^{\pm}) =\frac{2\pi}{\lambda} h(x)^{-\nicefrac{2}{3}} 
e^{\mp \nicefrac{\pi i}{3}} 
\exp\Big(-\frac{3}{8}A(q)\lambda\Big) \exp\Big(\pm\frac{3\sqrt{3}}{8}A(q)\, i\Big).
\]
Adding the two contributions we obtain that
$$
F(\lambda) \sim  \frac{4\pi}{\lambda} h(x)^{-2/3}
\exp\Big(-\frac{3}{8}A(q)\lambda\Big) \cos\Big(\frac{3\sqrt{3}}{8}A(q) -\frac{\pi}{3} \Big)
$$
It is straighforward to see that $p(x) = A(q)^{3/4}$ and applying \eqref{g:GF}
we therefore conclude that as $t\to 0+$,
\[
G(x,t) \sim \frac{2^{1/3}}{\pi} 
h(x)^{-2/3} t^{-1/3}\exp\Big(-\sigma_2 \frac{p(x)^{\frac{4}{3}}}{t^{\frac{1}{3}}} \Big) 
\cos\Big(\sigma_2\, \sqrt{3}\,  \frac{p(x)^{\frac{4}{3}}}{t^{\frac{1}{3}}}  -\frac{\pi}{3}\Big) .
\]

\

\

{\bf Data availability.} No data was used for the research described in this article.


\end{document}